\documentclass[11pt,a4paper]{article}
\usepackage{a4wide}
\usepackage{amsthm,amsmath,amssymb}
\usepackage{enumerate}
\usepackage{filecontents}
\usepackage{tikz}
\usepackage{bm}

\usetikzlibrary{decorations.pathreplacing}

\newtheorem{theorem}{Theorem}
\newtheorem{lemma}[theorem]{Lemma}

\theoremstyle{definition}
\newtheorem{definition}[theorem]{Definition}
\newtheorem{observation}[theorem]{Observation}

\newtheorem*{thmtightupper}{Theorem \ref{thm:super-tight-upper}}	
\newtheorem*{thmdigraphlower}{Theorem \ref{thm:digraph-lower}}	
\newtheorem*{thmincubator}{Theorem \ref{thm:incubator}}
\newtheorem*{thmsparseincubator}{Theorem \ref{thm:sparse-incubator}}
\newtheorem*{thmsparseUB}{Theorem \ref{thm:sparseUB}}

\newcommand{\Zzero}{\mathbb{Z}_{\geq 0}}  
\newcommand{\Zone}{\mathbb{Z}_{\geq 1}}

\newcommand{\pr}{\mathbb{P}}
\newcommand{\E}{\mathbb{E}}
 
\newcommand{\floor}[1]{\lfloor #1 \rfloor}
\newcommand{\ceil}[1]{\lceil #1 \rceil}
\let\epsilon=\varepsilon
 
\newcommand{\cc}{c}

\newcommand{\dout}{d_\textnormal{out}}
\newcommand{\din}{d_\textnormal{in}}
\newcommand{\Nout}{N_\textnormal{out}}
\newcommand{\Nin}{N_\textnormal{in}}
\newcommand{\Tend}{T_{\mathsf{end}}}
\newcommand{\Ti}[1]{T_{#1}}
\newcommand{\dd}{b}
\newcommand{\bb}{\dd}
\newcommand{\Tp}{T^+}
\newcommand{\Tpp}{T^{++}}
\newcommand{\Itp}[1]{\tau_{#1}}
\newcommand{\filt}{\mathcal{F}}

\newcommand{\tp}{t^+}

\newcommand{\Mp}{M^+}
\newcommand{\Mpp}{M^{++}}
\newcommand{\fail}{\mathsf{F}}
\newcommand{\yy}{\gamma}
\newcommand{\YY}{Y}
\newcommand{\stateY}{\mathcal{S}_Y}

\newcommand{\ZZ}{Z}
\newcommand{\zz}{z}	
\newcommand{\stateZ}{\mathcal{S}_Z}
\newcommand{\TZ}{T^Z}
\newcommand{\TY}{T^Y}
\newcommand{\down}{\mathsf{down}}
\newcommand{\up}{\mathsf{up}}	
\newcommand{\calE}{\mathcal{E}}
\newcommand{\ZTi}[1]{\mathcal{T}_{#1}}
\newcommand{\ZTend}{\mathcal{T}_{\mathsf{end}}}
\newcommand{\TT}{T}

\let\phi=\varphi

\usepackage{filecontents}

\begin{document}
\title{Asymptotically Optimal Amplifiers for the Moran Process}

\author{Leslie Ann Goldberg\thanks{University of Oxford, UK.
The research leading to these results has received funding from 
the European Research Council under the European Union's Seventh Framework Programme (FP7/2007-2013) ERC grant agreement no.\ 334828. The paper 
reflects only the authors' views and not the views of the ERC or the European Commission. The European Union is not liable for any use that may be made of the information contained therein.},  John Lapinskas\footnotemark[2], Johannes Lengler\thanks{ETH
Z\"urich, Switzerland.}, Florian Meier\footnotemark[3]\\Konstantinos Panagiotou\thanks{Ludwig-Maximilians-Universit\"at LMU M\"unchen, Germany} , and Pascal Pfister\footnotemark[3]}

\date{30 November 2017}   
\maketitle

\begin{abstract}   
We study the Moran process as adapted by Lieberman, Hauert and Nowak.
This is a model of an evolving population on a graph or digraph where certain 
individuals, called ``mutants'' have fitness $r$ and other individuals, called ``non-mutants'' have
fitness~$1$.
We focus on the situation where the mutation is advantageous, in the sense that $r>1$.
A family of digraphs is said to be strongly amplifying if the extinction probability
tends to~$0$ when the Moran process is run on digraphs in this family.
The most-amplifying known family of digraphs is the family of megastars of
Galanis et al. We show that this family is optimal, up to logarithmic factors, since every
strongly-connected $n$-vertex digraph has extinction probability $\Omega(n^{-1/2})$.
Next, we show that there is an infinite family of undirected graphs, called dense incubators, whose
extinction probability is $O(n^{-1/3})$. We show that this is optimal, up to constant factors.
Finally, we introduce sparse incubators, for varying edge density, and
show that the extinction probability of these graphs is $O(n/m)$, where $m$ is the number of edges.
Again, we show that this is optimal, up to constant factors.
 \end{abstract} 
     
\section{Introduction}
\label{sec:intro}
We study the Moran process~\cite{Moran}
as adapted by 
Lieberman, Hauert and Nowak~\cite{LHN, NowakBook}.
This is a model of an evolving population.
There are two kinds of individuals --- ``mutants'' and ``non-mutants''.
The model has
a parameter~$r$, which is a positive real number, and is the fitness of
the mutants. All non-mutants have fitness~$1$.  
The individuals reside at 
the vertices of a digraph~$G$ -- each vertex contains exactly 
one individual, and it is either a mutant or a non-mutant.
In the initial state,
one vertex (chosen uniformly at random)  contains a mutant.
All of the other vertices   contain non-mutants.
The process  evolves in discrete time.  At each step, a vertex
is selected at random, with probability proportional to its fitness.
Suppose that this is vertex~$v$. Next, an   out-neighbour $w$ of~$v$
is selected uniformly at random. Finally, the state of vertex~$v$ (mutant or non-mutant) 
is copied to vertex~$w$.

If $G$ is finite and strongly connected then with probability~$1$, the
process will  either reach the state where there are only mutants (known as
\emph{fixation}) or  it will reach the state where there are only non-mutants (\emph{extinction}).
If $G$ is not strongly connected then the process may continue changing forever ---
thus, it makes sense to restrict attention to strongly-connected digraphs~$G$.
We do so for the rest of the paper.

Given a strongly-connected digraph~$G$,
we use the notation $\rho_r(G)$
to denote the probability 
that the Moran
process (starting from a uniformly-chosen initial mutant) reaches fixation
and we use the notation $\ell_r(G)$ to denote the probability that it
reaches extinction.
If $\mathcal{G}$ is a set of digraphs
then we use $\ell_{r,\mathcal{G}}(n)$
to denote $\max\{ \ell_r(G) \mid \mbox{$G\in \mathcal{G}$ and $G$ has $n$ vertices}\}$.
(To avoid trivialities, we take the maximum of the empty set to be~$0$.)
The function $\ell_{r,\mathcal{G}}$ is called the
``extinction limit'' of the family~$\mathcal{G}$.
Lieberman et al.~\cite{LHN} raised the question of whether there exists an infinite family $\mathcal{G}$
of digraphs for which 
$\limsup_{n\rightarrow \infty}  \ell_{r,\mathcal{G}}(n) = 0$.
We say in this case that $\mathcal{G}$ is \emph{strongly amplifying}.
They defined two infinite families of strongly-connected digraphs --- superstars and metafunnels ---
which turn out to be strongly amplifying.
The  most  amplifying infinite family of strongly-connected digraphs that is known 
(in the sense that the extinction limit grows as slowly as possible, as a function of~$n$)
is the family $\Upsilon$ of \emph{megastars} from~\cite{UndirAmplifiers}.
Galanis et al.\ show \cite[Theorem 6]{UndirAmplifiers}
that, for every $r>1$ 
there is an $n_0$ (depending on~$r$) so that, for all $n \geq n_0$ and for every $n$-vertex digraph $G\in \Upsilon$,
$\ell_r(G) \leq {(\log n)}^{23} / n^{1/2}$.

The first question addressed by this paper is whether the family of megastars is optimal
in the sense that the extinction limit grows as slowly as possible (as a function of~$n$).
We show that this is the case, up to logarithmic factors.
\newcommand{\statethmdigraphlower}{For all $r>1$, any strongly-connected $n$-vertex digraph $G$ with $n \ge 2$ satisfies $\ell_r(G) > 1/(5rn^{1/2})$.}  
\begin{theorem}\label{thm:digraph-lower}
\statethmdigraphlower 
\end{theorem}   

For undirected graphs, the most amplifying graphs previously known were stars, whose extinction probability tends to $1/r^2$ (as the size of the star grows). In particular, no strongly-amplifying family of undirected graphs was known. In our next result we show that such families do exist, and that they can have extinction probability $\ell_r(G) = O(n^{-1/3})$.\footnote{See Section~\ref{sec:related} for a discussion of 
simultaneous independent work that also resolves this question.} Note that throughout the paper, we write ``graph'' exclusively to refer to undirected graphs, which we view as a special case of digraphs.

\newcommand{\denselowerbound}{71/(r(r-1)^2n)^{1/3}}
\newcommand{\statethmincubator}{For all $r > 1$, there exists an infinite family $\mathcal{D}_r$ of 
connected graphs with the following property. If $G \in \mathcal{D}_r$ has $n$ vertices, then 
$\ell_r(G) \leq \denselowerbound
$.}
\begin{theorem}\label{thm:incubator}
\statethmincubator
\end{theorem} 

The graphs in the family $\mathcal{D}_r$ are called \emph{dense incubators}.
Each such graph is parameterised by a number $k$, which is the
square of an integer.
Taking $\beta$ to be an integer constant depending on $r$, the graph consists of $k$ stars, each with $\lceil r\sqrt{\beta k}\rceil$ leaves, together with a clique of size $\beta k$.
Every centre of every star is connected to every node in the clique.
More details  are given in Definition~\ref{def:incubator} (this definition also defines sparse incubators, which
we will discuss shortly).

It is known \cite[Corollary 7]{MoranAbsorb}
that extinction probabilities are monotonic in~$r$
in the sense that if $0<r\leq r'$ then, for any digraph~$G$,
$\ell_{r'}(G) \leq \ell_{r}(G)$.
Thus, Theorem~\ref{thm:incubator} guarantees that, for every  $r'>r$ and every $n$-vertex graph in $\mathcal{D}_r$, we also have
$\ell_{r'}(G) \leq \denselowerbound$. 

The next question that we address is whether the family
$\mathcal{D}_r$
is optimal (again, in the sense that the extinction limit grows as slowly as possible).
We show that this is the case, up to constant factors (depending on~$r$). 
\newcommand{\statethmtightupper}{Let $r>1$. 
Consider any connected $n$-vertex graph $G$ with $n\geq 2$. Then
$\ell_r(G) >   1/(42r^{4/3} n^{1/3})$. }
\begin{theorem}\label{thm:super-tight-upper}
\statethmtightupper   \end{theorem}

The reason that dense incubators are called ``dense'' is that an $n$-vertex dense incubator has
$\omega(n)$ edges (more specifically, it has
$\Theta(n^{4/3}$) edges). 
The final question that we address is whether there are sparse families of graphs that
are strongly amplifying.
Once again, the answer is yes.

Before we present the relevant theorems (Theorems~\ref{thm:sparse-incubator} and~\ref{thm:sparseUB})
we define a (parameterised) family of incubators, where the additional parameter  
controls the edge density.
In order to define these, we need some definitions.
Given a graph $G=(V,E)$ and subsets $S$ and $T$ of $V$,
$E(S, T)$ denotes 
the set of edges in~$E$ with one endpoint in~$S$ and the other in~$T$.
We also use the following standard definition.

\begin{definition}\label{def:expander}
Let $G=(V,E)$ be a $d$-regular graph with $n$ vertices. $G$ is a \emph{small-set expander} if
\[
\min_{\substack{\emptyset \subset S \subseteq V,\\|S| \le n^{1/3}}}\frac{|E(S,V\setminus S)|}{|S|} \ge \frac{d}{4}.
\]
\end{definition}

 Let $\Zone$ denote the set of positive integers.
Given a graph $G=(V,E)$ and disjoint subsets $S$ and $T$ of $V$  
we use $G[S]$ to denote the subgraph of~$G$ induced by~$S$ and we use
$G[S,T]$ to denote the graph with vertex set $S\cup T$
and edge set $E[S,T]$. 
This graph is said to be \emph{biregular} if all vertices in~$S$ have the same degree
and also all vertices in~$T$ have the same degree.
Using Definition \ref{def:expander}, we can now define families of incubators (see Figure~\ref{fig:incubator}).

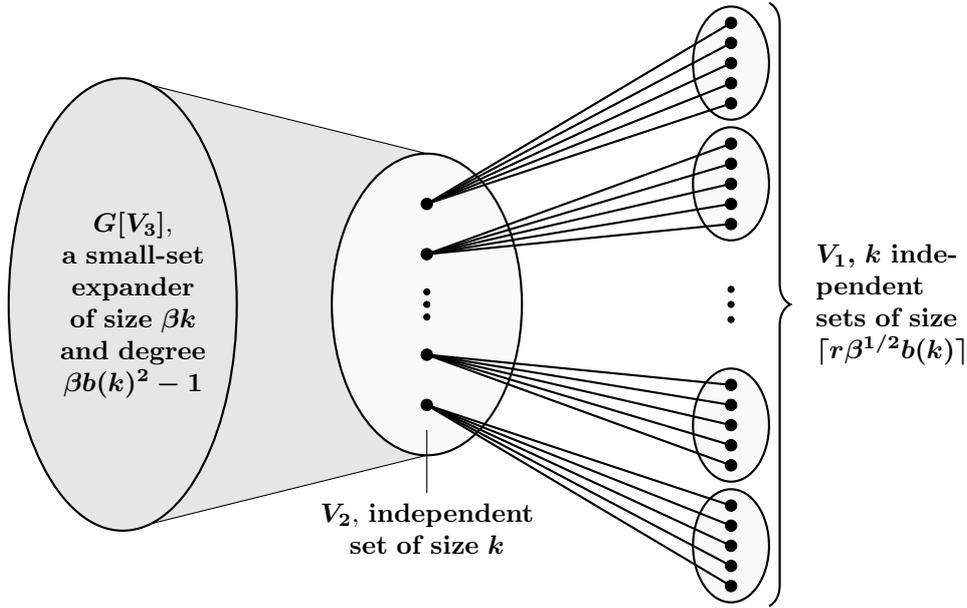
\begin{figure}\label{fig:incubator}
	\begin{center}
		\begin{tikzpicture}
		\colorlet{verylightgray}{gray!5};
		\colorlet{lightgray}{gray!20};
		\colorlet{anothergray}{gray!20};


\def\hr{3}	 		
\def\wr{1.5} 		
\def\s{4} 			
\def\t{4} 			
\def\hc{2} 			
\def\wc{1.25}  	
\def\hl{0.75}     	
\def\wl{0.5}		


\draw [fill=lightgray] (0,\hr) -- (0,-\hr) -- (\s,-\hc) -- (\s,\hc) -- (0,\hr);


\def\twr{1.5*\wr}
	
\draw [thick, fill=anothergray] (0,0) ellipse (\wr cm and \hr cm);

\draw  node[right, text width=\twr cm, align=center, font=\small] at (-\twr/2, 0) {$\bm{G[V_3]},$ \\ \textbf{a small-set expander of size} $\bm{\beta k}$ \textbf{and degree} $\bm{\beta b(k)^2 -1}$};

\def\twc{3*\wc}

\draw [thick, fill=verylightgray] (\s,0) ellipse (\wc cm and \hc cm);

\def\a{2*\hc / 6} 			
\foreach \y in {-\hc + \a, -\hc + 2*\a,-\hc + 4*\a,-\hc + 5*\a} {
		\draw [fill=black] (\s ,\y) circle(0.075cm); }

\foreach \y in {-\hc + 3*\a + \a/4,-\hc + 3*\a,-\hc + 3*\a - \a/4} {
		\draw [fill=black] (\s ,\y) circle(0.04cm); }

\draw node[below, text width=\twc cm, align=center, font=\small] at (\s, -\hc -0.5) {$\bm{V_2},$ \textbf{independent set of size} $\bm{k}$};

\draw (\s, -\hc -0.5) -- (\s, -\hc + 0.5*\a);		

		
\def\b{4*\hc /10} 			
\foreach \y in {-2*\hc + \b, -2*\hc + 3*\b,-2*\hc + 7*\b,-2*\hc + 9*\b} {
				\draw [thick, fill=verylightgray] (\s + \t,\y) ellipse (\wl cm and \hl cm); }
				
\foreach \y in {-2*\hc + 5*\b + \b/4,-2*\hc + 5*\b,-2*\hc + 5*\b - \b/4} {
		\draw [fill=black] (\s + \t ,\y) circle(0.04cm); }

\draw[thick,decorate,decoration={brace,amplitude=8pt}] (\s + \t + \wl , 2*\hc) -- (\s + \t + \wl , -2*\hc); 
			 
\draw node[right, text width=2cm, align=left, xshift=8pt, font=\small] at (\s + \t + \wl+0.2, 0) {$\bm{V_1, \, k}$ \textbf{independent sets of size} $\bm{\ceil{r \beta^{1/2} b(k)}}$};

\def\c{2*\b / 6}
			\foreach \y in {-2*\hc + \c, -2*\hc + 2*\c,-2*\hc + 3*\c,-2*\hc + 4*\c, -2*\hc + 5*\c} {
				\draw [fill=black] (\s + \t, \y) circle(0.075cm); 
				\draw [thick] (\s + \t, \y) -- (\s,-\hc + \a); }
				
			\foreach \y in {-2*\hc + 2*\b+ \c, -2*\hc + 2*\b + 2*\c,-2*\hc + 2*\b + 3*\c,-2*\hc + 2*\b + 4*\c, -2*\hc + 2*\b + 5*\c} {
				\draw [fill=black] (\s + \t, \y) circle(0.075cm); 
				\draw [thick] (\s + \t, \y) -- (\s,-\hc + 2*\a); }
				
		\foreach \y in {-2*\hc + 6*\b+ \c, -2*\hc + 6*\b + 2*\c,-2*\hc + 6*\b + 3*\c,-2*\hc + 6*\b + 4*\c, -2*\hc + 6*\b + 5*\c} {
				\draw [fill=black] (\s + \t, \y) circle(0.075cm); 
				\draw [thick] (\s + \t, \y) -- (\s,-\hc + 4*\a); }		
		
		\foreach \y in {-2*\hc + 8*\b+ \c, -2*\hc + 8*\b + 2*\c,-2*\hc + 8*\b + 3*\c,-2*\hc + 8*\b + 4*\c, -2*\hc + 8*\b + 5*\c} {
				\draw [fill=black] (\s + \t, \y) circle(0.075cm); 
				\draw [thick] (\s + \t, \y) -- (\s,-\hc + 5*\a); }

		\end{tikzpicture}
		\caption{The family of $\mathcal{I}_{r,b}$ incubators. As $G[V_2, V_3]$ is a biregular graph with $\beta k b(k)^2$ edges, each vertex in $V_2$ sends $\beta b(k)^2$ edges to $V_3$ and each vertex in $V_3$ sends $b(k)^2$ edges to $V_2$.}
	\end{center}
\end{figure}

\begin{definition}\label{def:incubator}
Let $r>1$ and let $\beta = 26\ceil{r^2/(r-1)}$.
Let $\dd:\Zone \rightarrow \Zone$ be any function
that satisfies $\dd(k) \le \sqrt{k}$ for all $k$. Then a graph $G = (V,E)$ is a member of the family $\mathcal{I}_{r,\dd}$ of \emph{incubators with branching factor $\dd$} if and only if there exists 
a positive integer~$k$ and a partition $V_1,V_2,V_3$ of $V$ such that the following properties hold.
\begin{enumerate}[(i)]
\item   $|V_1| = k\ceil{r\sqrt{\beta}\dd(k)}$, $|V_2| = k$, and $|V_3| = \beta k$.
\item   $G[V_1,V_2]$ is biregular with $k\ceil{r\sqrt{\beta}\dd(k)}$ edges.
\item   $G[V_2,V_3]$ is biregular with $\beta k\dd(k)^2$ edges.
\item   $G[V_1]$, $G[V_2]$, and $G[V_1,V_3]$ are empty.
\item   $G[V_3]$ is a small-set expander with degree $\beta \dd(k)^2 - 1$.
\end{enumerate}
\end{definition}
	
We will see at the end of this section 
how the branching factor~$\dd$ 
allows substantial control over the edge density of
incubators.
We will also see in Section~\ref{sec:incubators} (Theorem~\ref{lem:incubators-exist})
that, as long as $\dd(k)$ is eventually sufficiently large, 
then the set $\mathcal{I}_{r,\dd}$ is  infinitely large.
First, we present the relevant theorems.

\newcommand{\sparselowerbound}{2^{14}rn/((r-1)^2m)}
\newcommand{\statethmsparseincubator}{Let $r>1$.
There is a constant $b_0$ depending only on $r$ such that the following holds.
Let
$\dd:\Zone \rightarrow \Zone$ be any function
that satisfies $\dd(k) \le \sqrt{k}$ for all $k$.
Consider a graph
$G \in \mathcal{I}_{r,\dd}$  with branching factor $b(k) \geq b_0$.
Let $n$ be the number of vertices of $G$ and
$m$ be the number of edges of $G$.  Then 
$\ell_r(G) \leq \sparselowerbound$.
} 
\begin{theorem}\label{thm:sparse-incubator}
\statethmsparseincubator
\end{theorem}

As in the dense case, it turns out that the family $\mathcal{I}_{r,\bb}$ is optimal in 
the sense that (up to constant factors) the extinction limit grows as slowly as possible.

\newcommand{\statethmsparseUB}{Let $r>1$. Consider any 
connected graph $G$ with $n\geq 2$ vertices and $m$ edges.  Then 
$\ell_r(G) \ge {n}/(288 r^2m)$.}
\begin{theorem}\label{thm:sparseUB}
\statethmsparseUB
\end{theorem}

Theorem~\ref{thm:incubator} is closely related to a special case of Theorem~\ref{thm:sparse-incubator}.
To see this, consider the function $\dd(k)$ defined by
 $\dd(k) = \floor{\sqrt{k}}$.
Consider any $r>1$, and let $b_0$ be the constant (depending on~$r$) in 
Theorem~\ref{thm:sparse-incubator}. Let
\[
\mathcal{D}_r = \{G \in \mathcal{I}_{r,\dd} \mid \mbox{The parameter, $k$, of $G$ is the square of an integer and } \dd(k) \ge \dd_0\}.
\] 
Using these definitions, a slightly weaker version of Theorem~\ref{thm:incubator} 
may be obtained directly from the statement of Theorem~\ref{thm:sparse-incubator}. 
We give the proof of the stronger version in  Section~\ref{sec:main}.

Let us now consider sparse incubators.
In order to appreciate how the branching factor~$\dd$
controls the edge density of incubators, it is useful to calculate
the number of vertices and edges of a parameter-$k$ incubator in $\mathcal{I}_{r,b}$.

\begin{observation}\label{obs:incubator-nm}
Consider $G=(V,E)\in \mathcal{I}_{r,\dd}$ with  $|V_2|=k$, 
$|V|=n$ and $|E|=m$.
If $\dd(k)\geq \beta/r$ then
\begin{enumerate}[(i)]
\item $ k r\beta^{1/2} \dd(k) \le n \le  2 k r\beta^{1/2} \dd(k)$, 
\item $\beta^2 k\dd(k)^2/2 \le m \le \beta^2k\dd(k)^2$, and  
\item $   {\beta^{3/2} b(k)}/{(4r)} \leq  {m}/{n} \leq  {\beta^{3/2} b(k)}/{r}$.

\end{enumerate}
\end{observation}
\begin{proof} 
By the definition of $\mathcal{I}_{r,\dd}$, we have
$n = k\ceil{r\beta^{1/2}\dd(k)}+k+\beta k$ and
$m = k\ceil{r\beta^{1/2}\dd(k)}+\beta k\dd(k)^2 + \frac{k\beta(\beta \dd(k)^2-1)}{2}$.
The lower bounds on~$n$ and~$m$ follow immediately, and the upper bounds follow since $\beta \ge 26\ceil{r}$.
Putting these together gives the bounds on $m/n$.
\end{proof} 
  
Observation~\ref{obs:incubator-nm}
makes it easy to see that the function $\dd(k)$ can be tuned to achieve a variety of edge densities.

\subsection{Related work} \label{sec:related}

The Moran process is somewhat similar to a discrete version of
directed percolation known as the contact process. There  has been a lot of work
(e.g.,   \cite{Grimmett,DS,Durrett:NACS,
Lig1999:IntSys,Shah}) on the contact process and other related infection processes
such as the voter model and  SIS epidemic models. 
We refer the reader to \cite[Section 1.4]{UndirAmplifiers} for a discussion of how
these models differ from the Moran process.
 
Lieberman, Hauert and Nowak~\cite{LHN, NowakBook} 
introduced the version of the Moran process that we study. 
They raised the question of strong amplification and 
defined two infinite families of strongly-connected digraphs --- superstars and metafunnels ---
which turn out to be strongly amplifying. 
Many papers contributed to determining the fixation probability of
these  digraphs \cite{LHN,  DGMRSS2013:Superstars,JLH2015:Superstars} ---
 see \cite[Section 1.4]{UndirAmplifiers} for a discussion.
The first rigorous proof that there is an infinite family of strongly-amplifying digraphs is
in \cite{UndirAmplifiers}. This is the family of megastars discussed in the introduction.
The paper also gives lower bounds on the extinction probability of superstars and metafunnels.

The best-known lower bounds on the extinction probability of connected
undirected graphs are in  \cite{MSNatural, MSStrong}.
Theorem~1 of~\cite{MSStrong} shows that 
there is a constant $c_0(r)$ such that
for every $\epsilon>0$  
the extinction probability is at least $c_0(r)/n^{3/4+\epsilon}$.

While this manuscript was under preparation, 
George Giakkoupis posted simultaneous, independent work~\cite{George} 
also showing that strong undirected amplifiers exist. 
In the remainder of this section, we discuss this  work.

First, consider  the model of Lieberman, Hauert and Nowak~\cite{LHN, NowakBook}
which we study.
 Our Theorem~\ref{thm:incubator} 
shows that there is an infinite family of connected graphs $G$ with 
$\ell_r(G) \leq \denselowerbound$. 
Theorem~1 of~\cite{George} is
similar, but weaker by a  logarithmic factor ---  that paper constructs a (similar) family
with extinction probability $\ell_r(G) = O(\log(n)/((r-1) n^{1/3})$.
Our Theorem~\ref{thm:super-tight-upper} 
shows that any connected $n$-vertex graph (with $n\geq 3$)
has
$\ell_r(G) >   1/(42 r^{4/3} n^{1/3})$.
Theorem~2  of \cite{George} is similar, but weaker by a $(\log n)^{4/3}$ factor --- that paper
shows that the extinction probability 
$\ell_r(G)$ is $\Omega(1/(r^{5/3} n^{1/3} {(\log n)}^{4/3}))$.

Our paper is otherwise incomparable to \cite{George}.
We give a lower bound on the extinction probability of amplifying \emph{digraphs}
(Theorem~\ref{thm:digraph-lower}) but \cite{George} does not consider digraphs.
We also construct sparse families of incubators (Theorem~\ref{thm:sparse-incubator}) which 
go all the way down to constant density and
are optimally-amplifying
up to constant factors (Theorem~\ref{thm:sparseUB})  but \cite{George} does
not consider sparse graphs.
On the other hand, \cite[Theorem 3]{George} constructs a family of \emph{suppressors} 
with extinction probability at  least $1- O(r^2 \log n/n^{1/4})$, which is something that we do not study here.
Finally, Sood et al.~\cite{SAR} have introduced a variant of the model
in which the fitness of a mutant is taken to be a function of the number of vertices of the underlying digraph
(so as the number of vertices in the digraph grows, the fitness of each individual mutant decreases).
The results of \cite{George} extend to this model where $r=1+o(1)$, as a function of~$n$.
We are not aware of any applications of this model, and we don't consider it.

\subsection{Organisation of the paper}

In Section~\ref{sec:prelim}, we define some notation that we will use throughout the paper.
In Section~\ref{sec:incubators}, we 
show that, 
as long as $\dd(k)$ is eventually sufficiently large, 
then the set $\mathcal{I}_{r,\dd}$ is  infinitely large.
In Section~\ref{sec:main} we prove 
Theorems~\ref{thm:incubator} and~\ref{thm:sparse-incubator}, which give upper bounds on extinction probability.
In  Section~\ref{sec:lowerextinct}, we prove Theorems~\ref{thm:digraph-lower},
\ref{thm:super-tight-upper} and~\ref{thm:sparseUB} which give lower
bounds on extinction probability.

\section{Preliminaries}\label{sec:prelim} 
	 
We write $\Zone = \{1, 2, \dots\}$. For all $n \in \Zone$, we write $[n] = \{1, \dots, n\}$. We write $\log$ for the base-$e$ logarithm and $\lg$ for the base-2 logarithm. 
	
When $G = (V,E)$ is a digraph and $v \in V$, we write $\Nin(v) = \{w \mid (w,v) \in E\}$, $\din(v) = |\Nin(v)|$, $\Nout(v) = \{w \mid (v,w) \in E\}$, and $\dout(v) = |\Nout(v)|$. We view undirected graphs (or simply ``graphs'') as digraphs such that for all $u,v \in V$, $(u,v) \in E$ if and only if $(v,u) \in E$. Of course, we use standard conventions when counting edges in undirected graphs. That is, an undirected edge $\{u,v\}$ is only counted as one edge. If $G$ is undirected, we write $N(v) = \Nout(v) = \Nin(v)$ and $d(v) = \dout(v) = \din(v)$. 
If $S \subseteq V$, we write $N(S) = \bigcup_{v \in S}N(v)$.

Recall that the initial configuration of the Moran process is 
the configuration in which one vertex is chosen uniformly at random to be
a mutant, and the rest of the vertices are non-mutants.
We have already defined $\ell_r(G)$, which is the
probability that this process reaches extinction. 
When $G=(V,E)$ is known from the context and $v$ is a vertex of~$G$, it will also be
useful to define $\ell_r(v)$ to be the extinction probability, conditioned on the fact that the initial
mutant is~$v$.
In this case, $\ell_r(G) = \tfrac1n \sum_{v\in V} \ell_r(v)$.
More generally, when $G$ is known from the context 
and $U$ is a subset of $V(G)$, we define $\ell_r(U)$ to be the extinction probability,
when the process is run starting from the state in which vertices in~$U$ are mutants and vertices in $V\setminus U$ are non-mutants.

\section{Infinite sets of incubators}\label{sec:incubators}

The main result of this Section is  Theorem~\ref{lem:incubators-exist}, which shows
that, as long as $\dd(k)$ is eventually sufficiently large, 
then the set $\mathcal{I}_{r,\dd}$ is  infinitely large.
 
If a graph $G$ has adjacency matrix $A$ with eigenvalues $\lambda_1 \ge \dots \ge \lambda_n$, we 
let $\lambda(G) = \max\{\lambda_2, -\lambda_n\}$.
We use two existing results which, between them, imply that a sparse random regular graph
is likely to be a small-set expander.

\begin{theorem}\label{thm:cgj} (\cite[Theorem 1.1]{CGJ})
Let $C_0,K > 0$, and let $\alpha = 459652 + 229452K + \max\{30C_0^{3/2},768\}$. Let $n,d \in \Zone$, and suppose $d \le C_0 n^{2/3}$ and $n \ge 7+K^2$. Let $G$ be a uniformly random $d$-regular graph on $n$ vertices. Then $\pr(\lambda(G) \le \alpha\sqrt{d}) \ge 1 - n^{-K}$.\qed
\end{theorem}

The following theorem is well-known, and follows from, e.g., \cite[Theorem 8.6.30]{West}.

\begin{theorem}\label{thm:cheeger}
If $G=(V,E)$ is a $d$-regular $n$-vertex graph, and 
$S$ is a non-empty proper subset of $V$, then $|E(S,V\setminus S)| \ge (d-\lambda(G))\,|S|\,|V\setminus S|/n$.\qed
\end{theorem}

\begin{theorem}\label{lem:incubators-exist}
Let $\dd:\Zone \rightarrow \Zone$. Suppose that for all $k \in \Zone$, $\dd(k) \le \sqrt{k}$. Suppose in addition that 
there are infinitely many~$k$ such that
$\dd(k) \geq 10^{7}$.  Then $\mathcal{I}_{r,\dd}$ contains infinitely many graphs.
\end{theorem}

\begin{proof}
It suffices to prove that for all $k$ such that $\dd(k) \ge 10^{7}$, there exists a small-set expander $H_k$ on $\beta k$ vertices with degree $D = \beta \dd(k)^2 - 1$. 
Since $\dd(k) \le \sqrt{k}$, we have $0 \le D \le \beta k-1$. Since $\beta$ is even, it follows that there exists a regular graph with degree $D$ on $\beta k$ vertices. Let $H$ be a uniformly random such graph, and suppose 
that $\emptyset \subset S \subseteq V(H)$ satisfies $|S| \le (k\beta)^{1/3}$. 
		
\noindent {\bf Case 1.}\quad If $D \ge 2(k\beta)^{1/3}$, then we have
		\[
			|E(S,V(H)\setminus S)| \ge \sum_{v \in S}(d(v) - |S|) \ge D|S| - (k\beta)^{1/3}|S| > D|S|/4.
		\]
		Thus we may take $H_k = H$ with certainty.
		
\noindent {\bf Case 2.}\quad Suppose instead that $D \le  2(k\beta)^{1/3}$. 
Now apply Theorem~\ref{thm:cgj}
with $n=\beta k$, $d=D$, $C_0=2$ and $K=1$.
This shows that 
$\pr(\lambda(H) \le  \alpha \sqrt{d}) \ge 1 - 1/(\beta k)$.
If $\lambda(H) \leq \alpha \sqrt{d}$ then using Theorem~\ref{thm:cheeger}, 
$$ |E(S,V\setminus S)| \ge (D-  \alpha \sqrt{D} )\frac{|S|(k\beta-|S|)}{k\beta} \ge
(D-  \alpha \sqrt{D} ) \frac{|S|}{2}.$$
To show that $H$ is a small-set expander, we need only show that $\alpha \sqrt{D} \leq D/2$.
This follows  from the definitions of~$\alpha$ and $D$ using $\beta \geq 26$ and $b(k)\geq 10^{7}$.
Since $H$ is a small-set expander with non-zero probability, there exists
a small-set expander on $\beta k$ vertices with degree $D$, as required.
	\end{proof}

Theorem~\ref{lem:incubators-exist} shows that the set 
$\mathcal{I}_{r,\dd}$ is infinitely large, as long as there are infinitely many $k$ such that $\dd(k)\geq 10^{7}$.
This is sufficient for our purposes, since our goal is to show that
there are infinitely-many incubators, 
even with constant density.
However, the condition that $\dd(k) \geq 10^{7}$ is not necessary.
The lower bound could be weakened substantially by replacing the use of Theorem~\ref{thm:cgj} with the result of Friedman~\cite{Friedman}
when $b(k)<10^{7}$.

\section{Upper bounding the extinction probability of incubators}
\label{sec:main}
In this section, we prove Theorems~\ref{thm:incubator} and~\ref{thm:sparse-incubator}. 
For this, it will be useful to have a more formal definition of the Moran process, which defines some notation 
that we will use.
\begin{definition}\label{def:moran}
		Let $G = (V,E)$ be an $n$-vertex digraph, let $r > 1$, and let $x_0 \in V$. We define the \emph{Moran process $(X_t)_{t \ge 0}$ on $G$ with fitness $r$ and initial mutant $x_0$} inductively as follows, where all random choices are made independently. Let $X_0 = \{x_0\}$. For all $S \subseteq V$, let $W(S) = n + (r-1)|S|$. Given $X_t$ for some $t \ge 0$, we define $X_{t+1}$ as follows. Randomly choose a vertex $v_t \in V$ with distribution
		\[
			\pr(v_t = v) = 
				\begin{cases}
					r/W(X_t) & \textnormal{if }v \in X_t;\\ 
					1/W(X_t) & \textnormal{if }v \in V \setminus X_t.
				\end{cases}
		\]
		If $\dout(v_t) = 0$, then $X_{t+1} = X_t$. Otherwise, choose $w_t \in \Nout(v_t)$ uniformly at random. If $v_t \in X_t$, then $X_{t+1} = X_t \cup \{w_t\}$, and we say \emph{$v_t$ spawns a mutant onto $w_t$ at time $t+1$}. Otherwise, $X_{t+1} = X_t \setminus \{w_t\}$, and we say \emph{$v_t$ spawns a non-mutant onto $w_t$ at time $t+1$}.
	
		If there exists $t$ such that $X_t = \emptyset$, we say the process \emph{goes extinct} at time $t$, and if there exists $t$ such that $X_t = V$, we say the process \emph{fixates} at time $t$. In either case, we say the process \emph{absorbs} at time $t$.  \end{definition}
Note that Moran processes are discrete-time Markov chains, and that nothing happens if (e.g.) a mutant is spawned onto a mutant. The notation $v_1, v_2, \dots$ and $w_1, w_2, \dots$ is not used outside Definition~\ref{def:moran}. 

Incubators are defined in Definition~\ref{def:incubator} on page~\pageref{def:incubator}.
Whenever we discuss a specific graph $G=(V,E) \in \mathcal{I}_{r,\dd}$, we  will use the notation $V_1$, $V_2$, $V_3$, $k$ and $\beta$   from Definition~\ref{def:incubator} without explicitly redefining it.  
We use $\dd$ as shorthand for $\dd(k)$.
Because the final theorem assumes $\dd\geq b_0$
for a constant $b_0$, depending on~$r$,
it will do no harm to assume that $\dd$ is sufficiently large.
To avoid cluttering the notation, we assume $\dd\geq 6r$ everywhere, and
we mention it explicitly only if we use a stronger bound.  
For all $v \in V_1$, we write $\cc(v)$ for the (necessarily) unique neighbour of $v$ in $V_2$.

\subsection{Going from a mutant in $V_1$ to many mutants in $V_3$}

Our goal in this subsection is to prove Lemma~\ref{lem:phase-one}, which says that if the initial mutant is in $V_1$, then with probability at least roughly $1-1/b$ the process will eventually obtain $\floor{b^{1/3}}$ mutants in $V_3$.  To prove this lemma, we will couple the evolution of mutants in $V_3$ with the following Markov chain.

\begin{definition}\label{def:Y}
	Throughout the rest of the section, let $\yy = \floor{(k\beta)^{1/3}}$ and $r' = (1+r)/2$. Define $(\YY_t)_{t \ge 0}$ to be a discrete-time Markov chain with state space $\stateY = \{\mathsf{\fail},0,1,\dots,\yy+1\}$, initial state $0$, and the following transition matrix:
	\begin{alignat*}{8}
		&&p_{\fail,\fail} &= 1,&&&&&&\\
		&&p_{0,\fail} &= \frac{6}{r\beta^{1/2}\dd}, 
		&\quad p_{0,0} &= \left(1-\frac{6}{r\beta^{1/2}\dd}\right)\frac{1}{1+r'},
		&\quad p_{0,1} &= \left(1-\frac{6}{r\beta^{1/2}\dd}\right)\frac{r'}{1+r'},\\
		&\textnormal{for all }1 \le i \le \yy,\ &p_{i,\fail} &= \frac{10}{r\beta\dd^2},
		&\quad p_{i,i-1} &= \left(1 - \frac{10}{r\beta\dd^2}\right)\frac{1}{1+r'},
		&\quad p_{i,i+1} &= \left(1 - \frac{10}{r\beta\dd^2}\right)\frac{r'}{1+r'},\\
		&&p_{\yy+1,\yy+1} &= 1,&&&&&&
	\end{alignat*}
	and $p_{i,j} = 0$ for all other $i,j\in\stateY$.
\end{definition}

Our coupling is defined formally in Lemma~\ref{lem:Y-coupling} and will have the property that for all $i \ge 0$, if $Y_i \ne \mathsf{F}$, then there exists $t \ge 0$ such that $|X_t \cap V_3| \ge Y_i$. The ``failure state''~$F$
in the state space of~$Y$ is used in the coupling to capture the  possibility that the initial mutant (in~$V_1$) 
dies before the $\floor{b^{1/3}}$ mutants are obtained in~$V_3$.
Using the coupling, we prove  Lemma~\ref{lem:phase-one} by using standard techniques to show that $Y$ is likely to reach $\floor{b^{1/3}}$ before reaching the 
failure state $\mathsf{F}$. We next define some stopping times which will be important to the coupling. 
	
\begin{definition}\label{def:Ts}
Let 
\[\Tend = \min\{t \ge 0 \mid X_t \cap V_1 = \emptyset \textnormal{ or } |X_t \cap V_3| = \gamma+1\}.\]
Note that $\Tend$ is finite with probability~$1$.
Define $\Ti{0},\Ti{1},\dots$ recursively by $\Ti{0} = 0$ and
\[\Ti{i} = \min\left(\{\Tend\} \cup
\{t > \Ti{i-1} \mid X_t \cap V_3 \ne X_{t-1} \cap V_3 \}   \right).\]
\end{definition}
If $\Tend=0$ then $\Ti{0},\Ti{1},\dots$ are all~$0$.
Otherwise, with probability~$1$,
there is a $j$ such that 
$T_0 < \cdots <T_j$ and, for all $i\geq j$, $T_j=\Tend$. The $T_j$'s will be used as update times in our coupling, and $\Tend$ will be used as a decoupling time; if $T_j < \Tend$, then $Y_{j+1}$ will depend on $X_{T_{j+1}}$, and otherwise the two processes will evolve independently. Thus in order to construct the coupling, we will need assorted bounds on the behaviour of $X_{T_j}$ subject to $T_j < \Tend$.

We first deal with the case where $V_3$ contains at least one mutant at time $T_j < \Tend$. We require an upper bound on the probability that $X_{T_{j+1}} \cap V_1 = \emptyset$, and on the probability that $|X_{T_{j+1}} \cap V_3| = |X_{T_j} \cap V_3| - 1$. These bounds are given in Lemma~\ref{lem:non-empty-trans}. Proving them will require that $|X_{T_j} \cap V_3| \le \gamma$, which is true since $T_j < \Tend$. 
To simplify the presentation, we first do the relevant calculations
in the following two technical lemmas.

	\begin{lemma}\label{lem:non-empty-step}
		Let $t \ge 0$, let $M \subseteq V$, and suppose $1 \le |M \cap V_3| \le \gamma$. Then we have
		\begin{align*}
			\pr(|X_{t+1} \cap V_3| = |X_t \cap V_3| + 1 \mid X_t = M) 
			&\ge \frac{r|E(V_3 \cap M, V_3 \setminus M)|}{W(M)(\beta\dd^2+\dd^2-1)},\\
			\pr(|X_{t+1} \cap V_3| = |X_t \cap V_3| - 1 \mid X_t = M) 
			&\le \left(1 + \frac{5}{\beta}\right)\frac{|E(V_3 \cap M, V_3 \setminus M)|}{W(M)(\beta\dd^2+\dd^2-1)}.
		\end{align*}
		Moreover, if $M \cap V_1 \ne \emptyset$, then
		\[
			\pr(X_{t+1} \cap V_1 = \emptyset \mid X_t = M) 
			\le \frac{1}{W(M)\beta\dd^2}.
		\]
	\end{lemma}
	\begin{proof}
		For brevity, write $M' = M \cap V_3$. For the first equation, note that $|X_t \cap V_3|$ increases whenever a vertex in $M'$ spawns a mutant onto a vertex in $V_3 \setminus M'$. We therefore have
		\begin{align*}
			\pr(|X_{t+1} \cap V_3| = |X_t \cap V_3| + 1 \mid X_t = M) 
			&\ge \sum_{v \in M'}\frac{r}{W(M)}\cdot \frac{|N(v) \cap (V_3 \setminus M')|}{d(v)}\\
			&= \frac{r|E(M',V_3\setminus M')|}{W(M)(\beta\dd^2+\dd^2-1)},
		\end{align*}
		as required.

		For the second equation, note that $|X_t \cap V_3|$ decreases precisely when a vertex in $(V_2 \cup V_3)\setminus M$ spawns a non-mutant onto a vertex in $M'$. It follows that
		\begin{align*}
			\pr(|X_{t+1} \cap V_3| = |X_t \cap V_3| - 1 \mid X_t = M) 
			&= \sum_{v \in (V_2\cup V_3) \setminus M}\frac{1}{W(M)}\cdot\frac{|N(v) \cap M'|}{d(v)}\\
			&\le \frac{|E(V_2,M')|}{W(M)(\beta\dd^2+\ceil{r\beta^{1/2}\dd})}+\frac{|E(M', V_3 \setminus M')|}{W(M)(\beta\dd^2+\dd^2-1)}.
		\end{align*}
		Recall from Definition~\ref{def:incubator}(v) that $G[V_3]$ is a small-set expander, and $1 \le |M'| \le \gamma \le |V_3|^{1/3}$ by hypothesis, so since $\beta \ge 26$ we have 
\[
\frac{|E(M',V_3\setminus M')|}{\beta\dd^2+\dd^2-1} \ge 
\frac{(\beta b^2-1)|M'|}{4(\beta b^2+b^2-1)}\geq  \frac{|M'|}{5}.
\]
		Moreover, since every vertex in $V_3$ has degree $\dd^2$ into $V_2$,
		\[
		\frac{|E(V_2,M')|}{\beta\dd^2+\ceil{r\beta^{1/2}\dd}} \le \frac{\dd^2|M'|}{\beta\dd^2} = \frac{|M'|}{\beta}.
		\]
		It follows that
		\[
			\pr(|X_{t+1} \cap V_3| = |X_t \cap V_3| - 1 \mid X_t = M)
			\le \left(1 + \frac{5}{\beta}\right)\frac{|E(M',V_3\setminus M')|} {W(M)(\beta\dd^2+\dd^2-1)},
		\]
		as required. 
 		
For the third equation, recall that, by hypothesis, 
$M \cap V_1 \ne \emptyset$. 
For brevity,
write $p = \pr(X_{t+1} \cap V_1 = \emptyset \mid X_t = M)$. Note that if $|M \cap V_1| > 1$ then $p = 0$.
Suppose instead that $M \cap V_1 = \{v_0\}$ for some $v_0 \in V$. Note that $v_0$ can only become a non-mutant if its unique neighbour $c(v_0)$ spawns a non-mutant onto it. Thus, if $c(v_0) \in M$, $p = 0$. If $c(v_0) \notin M$, we have
		\[
			p = \frac{1}{W(M)}\cdot\frac{1}{\ceil{r\beta^{1/2}\dd}+\beta\dd^2} \le \frac{1}{W(M)\beta\dd^2}.
		\]
		The desired inequality therefore holds in all cases.
	\end{proof}

\begin{lemma}\label{lem:beta-algebra}
$$ \frac{1+5/\beta}{r+1+5/\beta} \leq \frac{1}{1+r'}- \frac{10}{r \beta b^2}.$$
\end{lemma}
\begin{proof} 
 $$
  \frac{1+5/\beta}{r+1+5/\beta}
\le \frac{1+5(r-1)/(26r^2)}{1+r} \le \frac{1}{1+r} + \frac{r-1}{5r^2(r+1)}.
$$
Moreover, since $r'=(r+1)/2$ and $\dd \ge 6$ we have
\begin{align*}
\frac{1}{1+r'} - \frac{1}{1+r} - \frac{10}{r\beta\dd^2}
&\ge \frac{2}{3+r} - \frac{1}{1+r} - \frac{10(r-1)}{900r^3}
\ge \frac{r-1}{(r+1)(r+3)} - \frac{r-1}{90r^2}\\
&\ge \frac{r-1}{(r+1)(r+3)} - \frac{4(r-1)}{45(r+1)(r+3)}\\
&= \frac{41(r-1)}{45(r+1)(r+3)}
\ge \frac{41(r-1)}{180r^2(r+1)} > \frac{r-1}{5r^2(r+1)}.
\end{align*} 
\end{proof}

	\begin{lemma}\label{lem:non-empty-trans}
		Let $t \ge 0$, let $j \ge 0$, let $M \subseteq V$, and suppose $1 \le |M \cap V_3| \le \gamma$ and $M \cap V_1 \ne \emptyset$. Then we have
		\begin{align*}
			\pr(X_{\Ti{j+1}} \cap V_1 = \emptyset \mid X_t = M, \Ti{j} = t \ne \Tend) &\le p_{1,\mathsf{F}}, \mbox{and}\\
			\pr(|X_{\Ti{j+1}} \cap V_3| = |X_{\Ti{j}} \cap V_3| - 1 \mid X_t = M, \Ti{j} = t \ne \Tend) &\le p_{1,0}.\\
		\end{align*}
	\end{lemma}
	
	\begin{proof}
Write $\filt$ for the event that $X_t = M$ and $\Ti{j} = t \ne \Tend$. Throughout, suppose that $\filt$ occurs. First note that this implies
		\begin{equation}\label{eqn:non-empty-trans-Ti}
			\Ti{j+1} = \min\{t' > t \mid X_{t'} \cap V_3 \ne X_t \cap V_3\textnormal{ or }X_{t'} \cap V_1 = \emptyset\}.
		\end{equation}
Now consider some  non-negative integer $i$. For $M' \subseteq V$, let
\[
q_i^{M'} = \pr(\Ti{j+1} = t+i+1 \textnormal{ and } X_{t+i} = M' \mid \filt).
\]
Let $\mathcal{M}_i = \{M' \subseteq V \mid q_i^{M'} \ne 0\}$. For all $M' \in \mathcal{M}_i$, let
		\begin{align*}
			p_{\mathsf{down},i}^{M'} &= \pr(|X_{t+i+1} \cap V_3| = |X_{t+i} \cap V_3| - 1 \mid X_{t+i} = M', \Ti{j+1} = t+i+1, \filt),\\
			p_{\mathsf{fail},i}^{M'} &= \pr(X_{t+i+1} \cap V_1 = \emptyset \mid X_{t+i} = M', \Ti{j+1} = t+i+1, \filt).
		\end{align*}
Note that if $M' \in \mathcal{M}_i$, then $M' \cap V_1 \ne \emptyset$. Also, $M' \cap V_3 = M \cap V_3$, so $1 \le |M' \cap V_3| \le \gamma$.
If $X_{t+i} = M'$, then the three events $|X_{t+i+1} \cap V_3| = |M' \cap V_3|+1$, $|X_{t+i+1} \cap V_3| = |M' \cap V_3|-1$ and $X_{t+i+1} \cap V_1 = \emptyset$ are disjoint, and (by \eqref{eqn:non-empty-trans-Ti}) conditioning on $\Ti{j+1} = t+i+1$ is precisely equivalent to conditioning on one of the three events occurring. For brevity, write $x(M') = |E(V_3 \cap M', V_3 \setminus M')|$ and $y = \beta\dd^2 + \dd^2 - 1$. It follows by Lemma~\ref{lem:non-empty-step} (with Lemma~\ref{lem:non-empty-step}'s $t$ equal to our $t+i$ and Lemma~\ref{lem:non-empty-step}'s $M$ equal to our $M'$) that for all $M' \in \mathcal{M}_i$,
		\begin{align*}
			p_{\mathsf{fail},i}^{M'}
			&\le \frac{1/(W(M')\beta\dd^2)}{rx(M')/(W(M')y)}
			= \frac{y}{rx(M')\beta\dd^2}.
		\end{align*}
		Recall that $G[V_3]$ is a small-set expander and, for $M' \subseteq \mathcal{M}_i$, $|M' \cap V_3| = |M \cap V_3| \le \gamma \le |V_3|^{1/3}$, and so 
		\[
			x(M') \ge (\beta\dd^2-1)|M' \cap V_3|/4 \ge \beta\dd^2|M\cap V_3|/5 \ge \beta\dd^2/5.
		\]
		Moreover, $y \le 2\beta\dd^2$. It follows that
		\begin{equation}\label{eqn:non-empty-trans-fail}
			p_{\mathsf{fail},i}^{M'} 
			\le \frac{10}{r\beta\dd^2}
			 = p_{1,\mathsf{F}}.
		\end{equation}
		
Similarly, it follows by Lemma~\ref{lem:non-empty-step} 
and Lemma~\ref{lem:beta-algebra} that 
\begin{align}\nonumber
p_{\mathsf{down},i}^{M'} 
&\le \frac{(1+5/\beta)x(M')/(W(M') y)}{rx(M')/(W(M') y) + (1+5/\beta)x(M')/(W(M') y)} \\\label{eqn:non-empty-trans-down}
&= \frac{1+5/\beta}{r+1+5/\beta}  \le \left(1 - \frac{10}{r\beta\dd^2}\right) \frac{1}{1+r'} = p_{1,0}.
\end{align}

		Now, by the law of total probability and \eqref{eqn:non-empty-trans-fail}, we have
		\[
			\pr(X_{\Ti{j+1}} \cap V_1 = \emptyset \mid \filt) 
			= \sum_{i \geq 0}\sum_{M' \in \mathcal{M}_i}p_{\mathsf{fail},i}^{M'}q_i^{M'}
			\le p_{1,\mathsf{F}}\sum_{i \geq 0}\sum_{M' \in \mathcal{M}_i}q_i^{M'}
			= p_{1,\mathsf{F}},
		\]
		and so the first equation of the lemma statement follows. Similarly, the second equation follows from \eqref{eqn:non-empty-trans-down}.
	\end{proof}

In Lemma~\ref{lem:non-empty-trans}, we gave the necessary bounds to ensure that our coupling works when $V_3$ contains at least one mutant at time $T_j < \Tend$. We now do the same when $V_3$ contains no mutants. Since in this case $T_{j+1}$ is the first time after $T_j$ at which $V_3$ contains a mutant or $V_1$ contains no mutants, we must show that $V_3$ is likely to gain a mutant after $T_j$. 
For this we will need the fact that $X_{T_j} \cap V_1 \ne \emptyset$, which holds because $T_j < \Tend$. We first prove the following ancillary lemma.

	\begin{lemma}\label{lem:empty-step}
		Let $t \ge 0$, let $M \subseteq V$, and suppose $M \cap V_1 \ne \emptyset$. Then there exists a stopping time $\Tpp > t$ such that the following hold:
		\begin{enumerate}[(i)]
			\item $\pr(\Tpp < \infty \mid X_t = M) = 1$;
			\item $\pr(X_{t'} \cap V_1 = \emptyset \textnormal{ for some }t < t' \le \Tpp \mid X_t = M) \le  1/(r\beta\dd^2)$;
			\item $\pr(|X_{\Tpp} \cap V_3| \ge 1 \mid X_t = M) \ge 1/(6\beta^{1/2}\dd)$.
		\end{enumerate}
	\end{lemma}
	\begin{proof}
		Let $v_0 \in V_1 \cap M$ be arbitrary, and recall that $c(v_0)$ is the unique neighbour of $v_0$. Let $\Tp$ be the minimum $t' > t$ such that at time $t'$, either $v_0$ spawns or $c(v_0)$ spawns a non-mutant onto $v_0$. Let $\Tpp$ be the minimum $t' > \Tp$ such that at time $t'$, either $c(v_0)$ spawns onto some vertex in $V_3$ or a neighbour of $c(v_0)$ spawns a non-mutant onto $c(v_0)$. Note that since each vertex in $V$ spawns infinitely often with probability 1, (i) holds.
		
		For all $i \geq 0$ and all $M_i \subseteq V$ with $v_0 \in M_i$, we have
		\begin{align*}
			\pr(v_0 \textnormal{ spawns at }t+i+1 \mid X_{t+i}=M_i) &= r/W(M_i),\\
			\pr(c(v_0)\textnormal{ spawns a non-mutant onto }v_0\textnormal{ at }t+i+1 \mid X_{t+i} = M_i) &\le 1/(\beta\dd^2W(M_i)).
		\end{align*}
		Write $p_i^{M'} = \pr(\Tp = t+i+1, X_{t+i} = M' \mid X_t = M)$. Then since $v_0 \in V_1 \cap M$ and $v_0$ remains a mutant throughout $\{t, t+1, \dots, \Tp-1\}$, it follows by the law of total probability applied to $\Tp$ that
		\begin{align}\nonumber
			\pr(v_0 \in X_{\Tp} \textnormal{ and }v_0 \textnormal{ spawns at }\Tp \mid X_t = M)
			&= \pr(v_0 \textnormal{ spawns at } X_{\Tp} \mid X_t = M)\\\nonumber
			&\ge \sum_{i \ge 0}\sum_{\substack{M' \subseteq V\\v_0 \in M'}} \frac{r/W(M')}{r/W(M')+1/(\beta\dd^2W(M'))}\cdot p_i^{M'}\\\label{eqn:empty-step-fail}
			&=\frac{r}{r+1/(\beta\dd^2)} \ge 1 - \frac{1}{r\beta\dd^2}.
		\end{align}
		If $v_0 \in X_{\Tp}$ and $v_0$ spawns at $\Tp$, it follows that $c(v_0) \in X_{\Tp}$. By the definition of $\Tpp$, it therefore follows that $c(v_0) \in X_{t'}$ for all $\Tp \le t' \le \Tpp-1$, and so $c(v_0)$ does not spawn a non-mutant onto $v_0$ at any time in $\{\Tp +1, \dots, \Tpp\}$. Hence, (ii) follows from \eqref{eqn:empty-step-fail}.
		
		Now, for all $i \geq 0$ and all $M_i \subseteq V$ with $c(v_0) \in M_i$, since $\beta \ge 26r$ and $\dd \ge 6r$ we have
		\begin{align*}
			\pr(c(v_0)\textnormal{ spawns into }V_3\textnormal{ at }t+i+1 \mid X_{t+i} = M_i) \ge \frac{r}{W(M_i)}\cdot \frac{\beta\dd^2}{\beta\dd^2+\ceil{r\beta^{1/2}\dd}} \ge \frac{r}{2W(M_i)}.
		\end{align*}
		Moreover, writing $\mathcal{E}$ for the event that some $v \in V \setminus X_{t+i}$ spawns onto $c(v_0)$ at time $t+i+1$, we have
		\[
			\pr(\mathcal{E} \mid X_{t+i} = M_i) \le \frac{\ceil{r\beta^{1/2}\dd}}{W(M_i)} + \frac{\beta\dd^2}{W(M_i)} \cdot \frac{1}{\beta\dd^2+\dd^2-1} \le \frac{r\beta^{1/2}\dd+1}{W(M_i)} + \frac{1}{W(M_i)} \le \frac{2r\beta^{1/2}\dd}{W(M_i)}.
		\]
		Now, suppose that $\Mp \subseteq V$ and $\tp > t$ are such that $\pr(X_{\tp} = \Mp \textnormal{ and }\Tp = \tp \mid X_t = M) \ne 0$ and $c(v_0) \in \Mp$. Write 
		\[
			q_i^{\Mpp} = \pr(\Tpp = \tp+i+1,X_{\tp+i} = \Mpp \mid \Tp = \tp, X_{\tp} = \Mp).
		\] 
		Since when $X_{\tp} = \Mp$ and $\Tp = \tp$, $c(v_0)$ remains a mutant throughout $\{\Tp, \dots, \Tpp-1\}$, by the law of total probability applied to $\Tpp$ it follows that
		\begin{align*}
			\pr(X_{\Tpp} \cap V_3 \ne \emptyset \mid \Tp = \tp, X_{\tp} = \Mp)
			&\ge \sum_{i \geq 0}\sum_{\substack{\Mpp \subseteq V\\c(v_0) \in \Mpp}} \frac{r/(2W(\Mpp))}{r/(2W(\Mpp)) + 2r\beta^{1/2}\dd/W(\Mpp)}q_i^{\Mpp}\\
			&= \frac{1}{1 + 4\beta^{1/2}\dd}
			\ge \frac{1}{5\beta^{1/2}\dd}.
		\end{align*}
		It therefore follows from the law of total probability applied to $\Tp$ and \eqref{eqn:empty-step-fail} that
		\begin{align*}
			\pr(X_{\Tpp} \cap V_3 \ne \emptyset \mid X_t = M) &\ge \sum_{\tp > t}\sum_{\substack{\Mp \subseteq V\\v_0 \in \Mp}}\pr(X_{\Tpp} \cap V_3 \ne \emptyset, \Tp = \tp, X_{\tp} = \Mp \mid X_t = M)\\
			&\ge \frac{1}{5\beta^{1/2}\dd}\pr(v_0 \in X_{\Tp} \mid X_t = M)
			\ge \frac{1}{5\beta^{1/2}\dd} - \frac{1}{r\beta\dd^2} \ge \frac{1}{6\beta^{1/2}\dd}.
		\end{align*}
		Hence (iii) follows.
	\end{proof}
	
	Lemma~\ref{lem:empty-trans} is then proved by repeatedly applying Lemma~\ref{lem:empty-step}.
	
\begin{lemma}\label{lem:empty-trans}
Let $t \ge 0$, let $j \ge 0$, let $M \subseteq V$, and suppose $M \cap V_1 \ne \emptyset$ and $M \cap V_3 = \emptyset$. Then 
$\pr(X_{\Ti{j+1}} \cap V_1 = \emptyset \mid X_t = M, \Ti{j} = t \ne \Tend) \le p_{0,\mathsf{F}}$.
\end{lemma}

\begin{proof}

Write $\filt$ for the event that $X_t = M$ and that $\Ti{j} = t \ne \Tend$, and suppose throughout that $\filt$ occurs. Thus, 
by definition, $\Ti{j+1}$ is precisely the earliest time $t' > t$ such that $X_{t'} \cap V_3 \ne \emptyset$ or $X_{t'} \cap V_1 = \emptyset$.

Define stopping times $\Itp{0}, \Itp{1}, \dots$ inductively as follows. Let $\Itp{0} = t$. If $\Itp{i} < \Ti{j+1}$, then we must have $X_{\Itp{i}} \cap V_1 \ne \emptyset$, so we define $\Itp{i+1}$ to be the stopping time $\Tpp$ obtained by applying Lemma~\ref{lem:empty-step} with $t = \Itp{i}$ and $M = X_{\Itp{i}}$. 
Note that in this case $\tau_{i+1} > \tau_i$.
If $\Itp{i} \ge \Ti{j+1}$, we set $\Itp{i+1} = \Itp{i}$. 

For every $i \geq 0$, every $M_i \subseteq V$ with $M_i \cap V_1 \neq \emptyset$ and $M_i \cap V_3 = \emptyset$, and every $t'\geq t$, write $\filt_{i,M_i,t'}$ for the event that $\tau_i=t'$, $X_{t'}=M_i$, $t'< T_{j+1}$ and $\filt$ occurs. By Lemma~\ref{lem:empty-step}, we have
$$
\pr(T_{j+1} \in (\tau_i,\tau_{i+1}] \mid \filt_{i,M_i,t'}) 
\ge \pr(|X_{\tau_{i+1}} \cap V_3| \ge 1 \mid \filt_{i,M_i,t'}) 
\ge 1/(6\beta^{1/2}b), $$
	and
$$
\pr(X_{T_{j+1}} \cap V_1 = \emptyset \mbox{ and } T_{j+1} \in (\tau_i,\tau_{i+1}]  \mid \filt_{i,M_i,t'}) 
\le 1/(r\beta b^2).$$
It follows that
\[
	\pr(X_{T_{j+1}} \cap V_1 = \emptyset \mid \filt_{i,M_i,t'} \mbox{ and } T_{j+1} \in (\tau_i,\tau_{i+1}]) \le \frac{1/(r\beta b^2)}{1/(6\beta^{1/2}b)} = p_{0,\mathsf{F}}.
\]
Thus, writing $\filt_{i,M_i,t'}'$ for the event that $\filt_{i,M_i,t'}$ occurs and $T_{j+1} \in (\tau_i,\tau_{i+1}]$, we have
\[
	\pr(X_{T_{j+1}} \cap V_1 = \emptyset \mid \filt) = \sum_{i,M_i,t'}\pr(X_{T_{j+1}} \cap V_1 = \emptyset \mid \filt_{i,M_i,t'}')\pr(\filt'_{i,M_i,t'} \mid \filt) \le p_{0,\mathsf{F}}
\]
which  proves the lemma.
\end{proof}
	
We now collect the bounds of Lemmas~\ref{lem:non-empty-trans} and~\ref{lem:empty-trans} into a single lemma.
	
\begin{lemma}\label{lem:coupling-probs}
Let $i \ge 0$, $t_i \ge 0$, $M \subseteq V$ and $y \ge 0$. Suppose that $y \le |M \cap V_3| \le \yy$ and $M \cap V_1 \neq \emptyset$. Write $\filt$ for the event that $X_{t_i} = M$ and $\Ti{i} = t_i \ne \Tend$. Then, we have
$	\pr(X_{\Ti{i+1}} \cap V_1 = \emptyset \mid \filt) \le p_{y,\fail}$ and
$$\pr(X_{\Ti{i+1}} \cap V_1 = \emptyset \mid \filt) + \pr(|X_{\Ti{i+1}} \cap V_3| = |X_{\Ti{i}} \cap V_3| - 1 \mid \filt) \le 1 - p_{y,y+1}.$$\end{lemma}
	
\begin{proof}
If $y > 0$, so that $M \cap V_3 \ne \emptyset$, then the result follows by Lemma~\ref{lem:non-empty-trans}. 
If instead $y = 0$ and $M \cap V_3 = \emptyset$, then the result follows by Lemma~\ref{lem:empty-trans} 
(using the observation that $\pr(|X_{\Ti{i+1}} \cap V_3| = |X_{\Ti{i}} \cap V_3| - 1 \mid \filt)=0$
and $p_{0,\fail} \leq 1-p_{0,1}$).
Finally, suppose $y = 0$ and $M \cap V_3 \ne \emptyset$. Then by Lemma~\ref{lem:non-empty-trans}, we have $\pr(X_{T_{i+1}} \cap V_1 = \emptyset \mid \filt) \le p_{1,\mathsf{F}} < p_{0,\mathsf{F}}$, and
		\begin{align*}
			&\pr(X_{\Ti{i+1}} \cap V_1 = \emptyset \mid \filt) + \pr(|X_{\Ti{i+1}} \cap V_3| = |X_{\Ti{i}} \cap V_3| - 1 \mid \filt)\le p_{1,\mathsf{F}} + p_{1,0} = 1 - p_{1,2} < 1 - p_{0,1}.
		\end{align*}
		Thus, the result follows in all cases.	
\end{proof}
	
We are now finally in a position to  define our coupling.
	
\begin{lemma}\label{lem:Y-coupling}
Suppose $X_0 \cap V_1 \ne \emptyset$. Then, there exists a coupling $\Phi(X,Y)$ between $(X_t)_{t \ge 0}$ and 
$(Y_t)_{t \ge 0}$ such that for all $i \ge 0$ with $Y_i \ne \fail$, there exists $t\le \Ti{i}$ such that $|X_t \cap V_3| \ge Y_i$.
\end{lemma}
\begin{proof}
We will construct a coupling $\Phi(X,Y)$ such that the following properties hold for every non-negative integer~$j$.
\begin{enumerate}
\item If $T_j < \Tend$, then either $Y_j=\fail$ or $|X_{T_j} \cap V_3| \ge Y_j$.
\item If $T_j = \Tend$ 
and 
$j = \min\{s\geq 0\mid T_s=\Tend\}$ and 
$|X_{T_j} \cap V_3| = \gamma+1$, then either $Y_j=\fail$ or $|X_{T_j} \cap V_3| \ge Y_j$. 
\item If $T_j=\Tend$ and 
$j=\min\{s\geq 0\mid T_s=\Tend\}$ and
$X_{T_j} \cap V_1 = \emptyset$, then $Y_j=\fail$.
\end{enumerate}

First, we observe that a coupling satisfying Properties 1--3 
would satisfy the condition in the statement of the lemma.
To see this, consider some non-negative integer $i$ for which we want to establish the condition 
in the statement of the lemma (that $Y_i=\fail$ or there exists $t\le \Ti{i}$ such that $|X_t \cap V_3| \ge Y_i$).
If $T_i<\Tend$ then this follows from Property~1 with $j=i$ and $t=T_i$.
If $T_i=\Tend$ 
and  
$i = \min\{s\geq 0\mid T_s=\Tend\}$,
then it follows from Properties~2 and~3 with $j=i$ and $t=T_i$.
Otherwise, there is a non-negative integer $j<i$ such that 
$T_j=\Tend$
and 
$j = \min\{s\geq 0\mid T_s=\Tend\}$.
Properties~2 and~3 guarantee that $Y_j=\fail$ (in which case the definition of $Y$ ensures that $Y_i=\fail$
so the condition is satisfied) or $|X_{T_j} \cap V_3| = \gamma + 1$ so (by the definition of $Y$) $Y_i \le |X_{T_j} \cap V_3|$ and taking $t=T_j$ satisfies the condition.

In order to construct the coupling, it will be useful to have some notation.
Given a coupling $\Phi(X,Y)$ and
a non-negative integer~$j$,
let $\Phi^j$ denote the initial sequence $(X_0,\ldots,X_{T_j},Y_0,\ldots,Y_j)$.
We will construct 
$\Phi(X,Y)$  by induction on $j$,
using $\Phi^j$ (and some randomness) to construct $\Phi^{j+1}$.
To do this, we have to ensure that Properties 1--3 are satisfied, and also 
that the coupling is valid, in the sense that
\begin{itemize}
\item The marginal distribution of $X_{{T_j}+1},\ldots,X_{T_{j+1}}$ is correct, given $X_{T_j}$
and given 
whether or not  
$T_j < \Tend$ (which can be deduced from $X_0,\ldots,X_{T_j}$), and
\item The marginal distribution of $Y_{j+1}$ is correct, given $Y_j$.
\end{itemize}

Note that $\Phi^0 = (X_0,0)$ satisfies Properties 1--3
(for Property~3 it is important that $X_0 \cap V_1 \neq \emptyset$ and this is guaranteed in the statement of the lemma)
so we now show how to construct $\Phi^{j+1}$, given $\Phi^j$.
In fact, if  $\Tend \leq T_j$ then
any coupling 
$\Phi(X,Y)$ which is consistent with $\Phi^j$ and   satisfies the two marginal distributions is fine 
(since the three properties are 
irrelevant for $T_i$ with $i>j$).
So we will not consider this case.
However, if $T_j < \Tend$  we will show how to construct $\Phi^{j+1}$.
\begin{itemize}
\item {\bf If $Y_j=\fail$:\quad} The definition of $Y$ guarantees that $Y_{j+1}=\fail$. This satisfies all three properties,
so let $X_{{T_j}+1},\ldots,X_{T_{j+1}}$
evolve independently of $Y_{j+1}$
according to its correct marginal distribution, 
given $X_{T_j}$
and given  the fact that $T_j < \Tend$.
\item {\bf If  $Y_j \neq \fail$:\quad}
Let $\calE_{\fail}$ be the event that
$X_{T_{j+1}}\cap V_1=\emptyset$
and let
$p_{\fail}$ denote the probability that $\calE_{\fail}$ occurs
in the correct marginal distribution (which depends only on $X_{T_j}$, noting that $T_j<\Tend$).
Let 
$\calE_{\down}$ be the event that 
$|X_{T_{j+1}}\cap V_3|= |X_{T_j} \cap V_3|-1$  
and let $p_{\down}$ be the probability that $\calE_{\down}$ occurs 
in the same marginal distribution.
Note that  $\calE_{\fail}$ and $\calE_{\down}$ are disjoint, since $T_j<\Tend$.
Let $\calE_{\up}$ be the event that
$|X_{T_{j+1}}\cap V_3|= |X_{T_j} \cap V_3|+1$.
In the marginal distribution, this occurs with probability 
$1-p_{\fail} - p_{\down}$.
By Property~$1$ and the definition of $\Tend$,
we have $0\leq Y_j \leq |X_{T_j} \cap V_3| \leq \gamma$. 
Now Lemma~\ref{lem:coupling-probs}
(with $i=j$, $t_i=T_j$, $M=X_{T_j}$ and $y=Y_j$) 
shows that  
$ p_{\fail} \le p_{Y_j,\fail}$ and
$p_{\fail} + p_{\down} \leq  1-p_{Y_j,Y_j+1}$.
The quantity $1-p_{Y_j,Y_j+1}$ is either
$p_{Y_j,\fail} + p_{Y_j,Y_j-1}$, if $Y_j>0$,
or $p_{Y_j,\fail} + p_{Y_j,Y_j}$, if $Y_j=0$.
To unify these cases, let $p_{Y_j,\down}$ be
$p_{Y_j,Y_j-1}$ if $Y_j>0$ and
$p_{Y_j,Y_j}$ if $Y_j=0$.
Then we have
\begin{equation}
\label{eq:ourprobs}
p_{\fail} \le p_{Y_j,\fail} \mbox{ and }
p_{\fail} + p_{\down} \leq p_{Y_j,\fail} + p_{Y_j,\down}.\end{equation} 
The coupling is as follows:  Choose
$X_{{T_j}+1},\ldots,X_{T_{j+1}}$ according to the correct marginal distribution.
\begin{itemize}
\item $\calE_{\fail}$ happens with probability $p_{\fail}\leq p_{Y_j,\fail}$. When this happens, set $Y_{j+1}=\fail$.
\item $\calE_{\up}$ happens with probability $p_{\up} \geq p_{Y_j,Y_{j}+1}$. 
When this happens, with probability $p_{Y_j,Y_{j}+1}/p_{\up}$, set $Y_{j+1}=Y_j+1$.
Let 
$\xi = p_{\up} - p_{Y_j,Y_j+1}$ and
$\rho = \min\{p_{Y_j,\down}, \xi\}$.
With probability $\rho/p_{\up}$, set $Y_{j+1}=\max\{Y_j-1,0\}$ and
with probability $(\xi - \rho)/p_{\up}$ set $Y_{j+1}= \fail$.
\item $\calE_{\down}$ happens with probability $p_{\down}$. When this happens,
let $\sigma = p_{Y_j,\down}- \rho$.
Let 
$Y_{j+1} = \max\{Y_{j}-1,0\}$ with  
probability $\sigma/p_{\down}$
 and $Y_{j+1} = \fail$, with  
 probability $1-\sigma/p_{\down}$.
\end{itemize}
It is now easy to check that 
$Y_{j+1}=Y_j+1$ and $Y_{j+1}=\max\{Y_j-1,0\}$ both happen with the correct marginal
distribution (so $Y_{j+1}=\fail$ does as well). 
Also,   equation~\eqref{eq:ourprobs} guarantees
that the probabilities are all well-defined (and non-negative!).
Finally, the coupling itself guarantees Properties 1--3.
\end{itemize}
\end{proof}	
	
Given our coupling, we want to analyse $(Y_t)_{t\ge 0}$, 
and this is easy, since it is very similar to the classical gambler's ruin problem. For completeness, we give details below.
	 	
	\begin{definition}\label{def:Z}
		Write $\stateZ = \{0, 1, \dots, \yy\}$, and suppose $\zz \in \stateZ$. Then we define $(\ZZ_t^z)_{t \ge 0}$ to be a discrete-time Markov chain with state space $\stateZ$, initial state $z$, and the following transition matrix:
		\begin{alignat*}{6}
			&&p_{0,0}' &= 1/(1+r'),
			&p_{0,1}' &= r'/(1+r'),\\
			&\textnormal{ for all }i \in [\yy-1],\ 
			&p_{i,i-1}' &= 1/(1+r'),\quad
			&p_{i,i+1}' &= r'/(1+r'),\\
			&&p_{\yy,\yy-1}' &= 1,
		\end{alignat*}
		and $p_{i,j}' = 0$ for all other $i,j\in\stateZ$.
	\end{definition}
	
	The following analysis of the classical gambler's ruin problem is well-known. See, for example, \cite[Chapter~XIV]{Fel1968:Probability}.
	
	\begin{lemma}\label{lem:gamblers-ruin}
		Consider a random walk on $\Zzero$ that absorbs at $0$ and $a$ (for some positive integer $a$), starts at $z \in \{0, \dots, a\}$, and from each state in $\{1, \dots, a-1\}$ has probability $p \ne 1/2$ of increasing (by 1) and probability $q=1-p$ of decreasing (by 1). Then, the probability of reaching state $a$ is
		\[
			\frac{1 - (q/p)^z}{1 - (q/p)^a}.
		\]
		Moreover, if $p > 1/2$, then the expected number of transitions before absorption is at most
		\[
			\frac{a}{p-q}\cdot \frac{1-(q/p)^z}{1-(q/p)^a}. 
		\]\qed
	\end{lemma}

\begin{lemma}\label{lem:Z-Y-props}
Suppose  
$\dd \ge {((1/\lg r')+1)}^3$, 
and write
$\TZ = \min \{i \mid \ZZ^0_i = \floor{\dd^{1/3}}\}$. 
Then 
$$\E(|\{0 \le i < \TZ \mid \ZZ^0_i = 0\}|) \le 2r'/(r'-1), \mbox{and}$$
$$\E(\TZ) \le  6 \floor{\dd^{1/3}}r'/(r'-1).$$ 
\end{lemma}
\begin{proof}
By Lemma~\ref{lem:gamblers-ruin}, the probability of reaching $\floor{\dd^{1/3}}$ before $0$ in $\ZZ^1$   is
\[\frac{1 - 1/r'}{1 - (1/r')^{\floor{\dd^{1/3}}}} \ge \frac{r'-1}{r'}.\]
Thus in $Z^0$, the probability of reaching $\floor{\dd^{1/3}}$ 
before returning to $0$ is at least $p_{0,1}'(r'-1)/r' = (r'-1)/(r'+1)$. Thus, the number of steps $\ZZ^0$ spends at $0$ before reaching $\floor{\dd^{1/3}}$ is dominated from above by a geometric variable with parameter $(r'-1)/(r'+1)$, and so
\[\E(|\{0 \le i < \TZ \mid \ZZ^0_i = 0\}|) \le \frac{r'+1}{r'-1} \le \frac{2r'}{r'-1},\]
as required.
		
Now, by Lemma~\ref{lem:gamblers-ruin}, the expected number of transitions  
that it takes for $\ZZ^1$ to reach either~$0$
or $\floor{\dd^{1/3}}$ is at most 
\[
\frac{\floor{\dd^{1/3}}(r'+1)}{r'-1}\cdot \frac{1 - 1/r'}{1-(1/r')^{\floor{\dd^{1/3}}}} \le 
\frac{\floor{\dd^{1/3}}(r'+1)}{r'-1}\cdot2(1-1/r') = \frac{2\floor{\dd^{1/3}}(r'+1)}{r'}.
\] 
So  
the expected number of transitions that it takes $\ZZ^0$ to
return to~$0$ or reach
$\floor{\dd^{1/3}}$ is at most 
$$1 + p_{0,1}'\cdot\frac{2\floor{\dd^{1/3}}(r'+1)}{r'} = 1 + 2\floor{\dd^{1/3}} \le 3\floor{\dd^{1/3}}.$$
By Wald's equation, it follows that
\[
\E(\TZ) \leq \left(\frac{2r'}{r'-1}\right)3\floor{\dd^{1/3}},
\]
and so the result follows.
\end{proof}
	
\begin{lemma}\label{lem:Y-hit}
Suppose $\dd \ge \max\{
{((1/\lg r')+1)}^3, 
120 
\}$. Then the probability that $\YY$ reaches state $\floor{\dd^{1/3}}$ is at least 
$1 - 25/(\beta^{1/2}\dd(r-1))$.
\end{lemma}
	
\begin{proof}
Let $\TZ = \min \{i \mid \ZZ^0_i = \floor{\dd^{1/3}}\}$ and let 
$\TY = \min\{i \ge 0 \mid \YY_i \in \{\floor{\dd^{1/3}}, \fail\}\}$. Then we have
		\begin{align*}
			\pr(\YY_{\TY} = \fail) &= \sum_{i=0}^{\infty} \sum_{x=0}^{\floor{\dd^{1/3}}-1} \pr(\YY_i = x\textnormal{ and }\TY > i)p_{x,\fail}\\
			&= p_{0,\fail}\sum_{i=0}^{\infty} \pr(\YY_i = 0\textnormal{ and }\TY > i) + p_{1,\fail}\sum_{i=0}^{\infty} \sum_{x=1}^{\floor{\dd^{1/3}}-1} \pr(\YY_i = x\textnormal{ and }\TY > i).
		\end{align*}
Since $\floor{\dd^{1/3}} \leq \gamma$, the definitions of~$Y$ and $Z^0$
show that	  the following are equivalent.
\begin{itemize}
\item   $(y_0, \dots, y_i)$ is a possible value of $(\YY_0, \dots, \YY_i)$ which implies $\YY_i = x$ and $\TY > i$.
\item $(y_0, \dots, y_i)$   is a possible value of $(\ZZ^0_0, \dots, \ZZ^0_i)$ which implies $\ZZ^0_i = x$ and $\TZ > i$. 
\end{itemize}

Moreover, for all $0 \le i \le \floor{\dd^{1/3}}-1$ and all $0 \le j \le \floor{\dd^{1/3}}$ we have $p_{i,j} \le p_{i,j}'$. It follows that
\begin{align*}
\pr(\YY_{\TY} = \fail) &\le p_{0,\fail}\sum_{i=0}^{\infty} \pr(\ZZ_i^0 = 0\textnormal{ and }\TZ > i) + p_{1,\fail}\sum_{i=0}^{\infty} \sum_{x=1}^{\floor{\dd^{1/3}}-1} \pr(\ZZ^0_i = x\textnormal{ and }\TZ > i)\\
&\le p_{0,\fail}\cdot\E(|\{0 \le i < \TZ \mid \ZZ^0_i = 0\}|) + p_{1,\fail}\cdot \E(\TZ).
\end{align*}
It follows by Lemma~\ref{lem:Z-Y-props} and the fact that $\dd\ge  
120$ (and hence $\dd^{2/3} \ge  
120/5$) that 	
\begin{align*}
\pr(\YY_{\TY} = \fail) &\le \frac{6}{r\beta^{1/2}\dd}\cdot \frac{2r'}{r'-1} + \frac{10}{r\beta\dd^2}\cdot 
\frac{ 
6\floor{\dd^{1/3}}r'}{r'-1} \le \frac{r'}{r(r'-1)\beta^{1/2}\dd}\left(12 + \frac{ 
60}{\dd^{2/3}\beta^{1/2}} \right)\\
&\le \frac{2}{(r-1)\beta^{1/2}\dd}\left(12 + \frac{1}{2}\right)
\le \frac{25}{\beta^{1/2}\dd(r-1)},
\end{align*}
as required.
\end{proof}

The goal of the subsection, Lemma~\ref{lem:phase-one}, now follows easily from Lemmas~\ref{lem:Y-coupling} and~\ref{lem:Y-hit}.
 
\begin{lemma}\label{lem:phase-one}
Suppose $\dd \ge \max\{
{((1/\lg r')+1)}^3, 
120 
\}$
and $X_0 \cap V_1 \ne \emptyset$. Then with probability at least 
$1 - 25/(\beta^{1/2}\dd(r-1))$, there exists $t \ge 0$ such that $|X_t \cap V_3| \ge \floor{\dd^{1/3}}$.
\end{lemma}
\begin{proof}
By Lemma~\ref{lem:Y-hit}, the probability that $\YY$ reaches state $\floor{\dd^{1/3}}$ is at least
$1 - 25/(\beta^{1/2}\dd (r-1))$. The result therefore follows from Lemma~\ref{lem:Y-coupling}.
\end{proof}

	\subsection{Going from mutants in $V_3$ to fixation}
	
Our goal in this subsection is to prove Lemmas~\ref{lem:fixate-from-leaf} and~\ref{lem:fixate-from-clique}, which give lower bounds on fixation probability conditioned on $X_0 \subseteq V_1$ and $X_0 \subseteq V_3$ respectively. Both arguments rely heavily on Lemma~\ref{lem:fixate-from-full} below, which says that fixation is very likely if $V_3$ contains at least $\floor{b^{1/3}}$ mutants. (Indeed, Lemma~\ref{lem:fixate-from-leaf} is immediate from this combined with Lemma~\ref{lem:phase-one}.) To prove Lemma~\ref{lem:fixate-from-full}, as in the previous section, we will need to couple the evolution of mutants in $V_3$ with a gambler's ruin. However, this time we will need the coupling to last until the gambler's ruin absorbs --- we cannot afford a chance of failure at every transition.
\begin{lemma}\label{lem:Z-coupling}
Suppose $t \ge 0$ and $M \subseteq V$. Let $z \ge 1$, and suppose $z \le \min\{|M \cap V_3|, \yy\}$. 
Let $I = \min\{i \mid \ZZ_i^z = 0\}$.
Then, conditioned on $X_t = M$, there exists a coupling $\Psi(X,\ZZ^z)$ between $(X_{t'})_{t' \ge t}$ and $(\ZZ_{t'}^z)_{t'\ge 0}$  
such that, for all $i< I$, 
there is a $t' \geq t+i-1$ such that
$|X_{t'} \cap V_3| \geq Z_i^z$.
\end{lemma}
\begin{proof}
Following  Definition~\ref{def:Ts}, let
$\ZTend = \min\{\tilde{t} \ge t \mid X_{\tilde{t}} \cap V_3 = \emptyset  \mbox{ or }
X_{\tilde{t}}=V\}$. 
Note that $\ZTend$ is finite with probability~$1$. Define $\ZTi{0}, \ZTi{1}, \dots$ recursively by $\ZTi{0} = t$ and
		\[
		\ZTi{i} = \min(\{\ZTend\}\cup \{\tilde{t} > \ZTi{i-1} \mid X_{\tilde{t}} \cap V_3 \ne X_{\tilde{t}-1} \cap V_3\}).
		\]
		
Consider any $t_i \ge t$ and $M_i \subseteq V$ with $1 \le |M_i \cap V_3| \le \yy$. Write $\filt_i$ for the event that $\ZTi{i} = t_i \ne \ZTend$, $X_{t_i} = M_i$, and $X_t = M$. For $t_i' \ge t_i$ and $M_i' \subseteq V$, write $$p_{t_i'}^{M_i'} = \pr(X_{t_i'} = M_i'\textnormal{ and }\ZTi{i+1} = t_i'+1 \mid \filt_i).$$	
Then we have
		\begin{align*}
		&\pr(|X_{\ZTi{i+1}} \cap V_3| = |X_{\ZTi{i}} \cap V_3| - 1 \mid \filt_i)\\
		=\ &\sum_{t_i' \ge t_i}\sum_{\substack{M_i' \subseteq V\\M_i' \cap V_3 = M_i \cap V_3}} \pr(|X_{t_i'+1} \cap V_3| = |X_{t_i'}\cap V_3|-1 \mid X_{t_i'} = M_i',\ZTi{i+1} = t_i'+1,\filt_i)\cdot p_{t_i'}^{M_i'}.
		\end{align*}
		Note that since $1 \le |M_i' \cap V_3| \le \yy \le |V_3|-1$, the conditioning on $\ZTi{i+1} = t_i'+1$ in the above expression is precisely equivalent to conditioning on $|X_{t_i'+1} \cap V_3| = |X_{t_i'} \cap V_3| \pm 1$. Moreover, by Lemma~\ref{lem:non-empty-step}, writing $\kappa = |E(V_3 \cap M_i', V_3 \setminus M_i')|/(W(M_i')(\beta\dd^2+\dd^2-1))$, when $M_i' \cap V_3 = M_i \cap V_3$ we have
		\begin{align*}
		\pr(|X_{t_i'+1} \cap V_3| = |X_{t_i'} \cap V_3| + 1 \mid X_{t_i'} = M_i', \filt_i) &\ge r\kappa,\\
		\pr(|X_{t_i'+1} \cap V_3| = |X_{t_i'} \cap V_3| - 1 \mid X_{t_i'} = M_i', \filt_i) &\le (1+5/\beta)\kappa.
		\end{align*}
It therefore follows from Lemma~\ref{lem:beta-algebra}
that 
		\begin{equation}\label{eqn:Z-coupling-probs}
		\pr(|X_{\ZTi{i+1}} \cap V_3| = |X_{\ZTi{i}} \cap V_3| - 1 \mid \filt_i) \le \frac{1+5/\beta}{1+5/\beta+r}\sum_{t_i' \ge t_i}\sum_{\substack{M_i' \subseteq V\\M_i' \cap V_3 = M_i \cap V_3}}p_{t_i'}^{M_i'} \le \frac{1}{1+r'}.
		\end{equation}
		
		Let $I' = \min\{i \mid \ZTi{i} = \ZTend\}$. We are now in a position to define a coupling $\Psi(X, \ZZ^z)$ such that 
		\begin{equation}\label{eqn:Z-coupling-prop}
		\textnormal{for all }j \le I',\ |X_{\ZTi{j}} \cap V_3| \ge \ZZ_j^z.
		\end{equation}
We first observe that such a coupling would satisfy the condition in the statement of the lemma.
Consider $i<I$. We wish to show that 
there is a $t' \geq t+i-1$ such that
$|X_{t'} \cap V_3| \geq Z_i^z$. 
There are two cases to consider. 
\begin{itemize}
\item
If  $i < I'$, then, since $\ZTi{i} < \ZTend$, we have $t+i -1 < \ZTi{i}$.
But from~\eqref{eqn:Z-coupling-prop},
$ |X_{\ZTi{i}} \cap V_3| \ge \ZZ_i^z$. So we can take $t'=\ZTi{i}$.
\item
Suppose instead that $i\geq I'$.   
From the definition of~$I'$, $\ZTi{I'} = \ZTend$, so
from~\eqref{eqn:Z-coupling-prop}, 
we have $|X_{\ZTend} \cap V_3| \ge \ZZ_{I'}^z$.
But since $I'\leq i< I$, $\ZZ_{I'}^z>0$, so
since $X_{\ZTend} \cap V_3$ is non-empty,
the definition of $\ZTend$ implies that the process fixates by time~$\ZTend$.
Thus, for any $t' \geq \ZTend$, we have
$|X_{t'} \cap V_3| =  |V_3|$, and this is at least $Z_j^z$ for any $j$ (since the state space of $Z^z$ only 
goes up to~$\gamma$) so it suffices to take any $t'\geq \max\{\ZTend, t+i-1\}$.  
\end{itemize}
		
		Given a coupling $\Psi(X,\ZZ^z)$ and a non-negative integer $j$, let $\Psi^j$ denote the initial sequence $(X_0, \dots, X_{\ZTi{j}}, \ZZ^z_0, \dots, \ZZ^z_j)$. We will first construct the sequence $\Psi^0, \Psi^1, \dots$ by induction on $j$, using $\Psi^j$ (and some randomness) to construct $\Psi^{j+1}$. We will continue this process until, for some $j > 0$, we obtain a $\Psi^j$ which implies $\ZTi{j} = \Tend$. (Note that $\pr(I' < \infty \mid X_t = M) = 1$.) We will then complete the coupling by allowing $X_{\ZTend+1}, X_{\ZTend+2}, \dots$ and $Z_{I'+1}, Z_{I'+2}, \dots$ to evolve independently according to their marginal distributions (which will vacuously satisfy \eqref{eqn:Z-coupling-prop}). Note that $\Psi^0 = (M,z)$ satisfies \eqref{eqn:Z-coupling-prop} since $z \le |M \cap V_3|$. Suppose we are given $\Psi^j$ satisfying \eqref{eqn:Z-coupling-prop} with $\ZTi{j} < \ZTend$. We will now construct $\Psi^{j+1}$.

		\begin{itemize}
			\item \textbf{If $\boldsymbol{|X_{\ZTi{j}} \cap V_3| \ge \yy+1}$:\quad} We let $X_{\ZTi{j}+1}, \dots, X_{\ZTi{j+1}}$ and $\ZZ^z_{j+1}$ evolve independently according to their correct marginal distributions. Note that $|X_{\ZTi{j+1}} \cap V_3| \ge |X_{\ZTi{j}} \cap V_3| - 1 \ge \yy \ge \ZZ^z_{j+1}$, so \eqref{eqn:Z-coupling-prop} is satisfied for $j+1$.
			
			\item  \textbf{If $\boldsymbol{\ZZ^z_j = \yy}$:\quad} We let $X_{\ZTi{j}+1}, \dots, X_{\ZTi{j+1}}$ and $\ZZ^z_{j+1}$ evolve independently according to their correct marginal distributions. Note that by \eqref{eqn:Z-coupling-prop}, $|X_{\ZTi{j+1}} \cap V_3| \ge |X_{\ZTi{j}} \cap V_3| -1 \ge \ZZ^z_j-1 = \ZZ^z_{j+1}$, so \eqref{eqn:Z-coupling-prop} is again satisfied for $j+1$.
			
			\item \textbf{If $\boldsymbol{|X_{\ZTi{j}} \cap V_3| \le \yy}$ and $\boldsymbol{\ZZ^z_j < \yy}$:\quad} Note that since $\ZTi{j} < \ZTend$, in this case we also have $|X_{\ZTi{j}} \cap V_3| \ge 1$. Let $\mathcal{E}_{\down}$ be the event that $|X_{\ZTi{j+1}} \cap V_3| = |X_{\ZTi{j}} \cap V_3| - 1$, and let $p_{\down}$ be the probability that $\mathcal{E}_{\down}$ occurs in the correct marginal distribution (which depends only on $X_{\ZTi{j}}$). Let $\mathcal{E}_{\up}$ be the event that $|X_{\ZTi{j+1}} \cap V_3| = |X_{\ZTi{j}} \cap V_3| + 1$, and let $p_{\up} = 1 - p_{\down}$. Then \eqref{eqn:Z-coupling-probs} shows that $p_{\down} \le 1/(1+r')$. The coupling is as follows. 
			\begin{itemize}
				\item Choose $X_{\ZTi{j}+1}, \dots, X_{\ZTi{j+1}}$ according to the correct marginal distribution.
				\item $\mathcal{E}_{\down}$ occurs with probability $p_{\down} \le 1/(1+r')$. When this happens, set $\ZZ^z_{j+1} = \max\{\ZZ^z_j-1,0\}$.
				\item $\mathcal{E}_{\up}$ occurs with probability $p_{\up} \ge r'/(1+r')$. When this happens, set $\ZZ^z_{j+1} = \ZZ^z_j+1$ with probability $r'/(p_{\up}(1+r'))$, and set $\ZZ^z_{j+1} = \max\{\ZZ^z_j-1,0\}$ with probability $1-r'/(p_{\up}(1+r'))$.
			\end{itemize} 
			It is now easy to check that $\ZZ^z_{j+1} = \ZZ^z_j+1$ and $\ZZ^z_{j+1} = \max\{\ZZ^z_j-1,0\}$ both happen with the correct marginal distribution. Moreover, the coupling itself guarantees that $|X_{\ZTi{j+1}} \cap V_3| \ge \ZZ^z_{j+1}$, so \eqref{eqn:Z-coupling-prop} is satisfied for $j+1$.
		\end{itemize}
	\end{proof}
	
We  will use the following Lemma from \cite[Theorem 9]{DGMRSS2014:approx} (which applies to all
graphs).	
\begin{lemma}\label{lem:absorb-time}
For all $M \subseteq V$, the expected absorption time of $X$ from state $M$ is at most $r|V|^4/(r-1)$. \qed
\end{lemma}

Lemma~\ref{lem:absorb-time}   implies that, in order to prove that the Moran process is likely to fixate, 
it suffices to show that it runs for a long time without going extinct.
We will use this in the proof of Lemma~\ref{lem:fixate-from-full}.
 
\begin{lemma}\label{lem:fixate-from-full}
There exists $\dd_0$ depending only on $r$ such that the following holds whenever $\dd \ge \dd_0$. Suppose $t \ge 0$ and $M \subseteq V$ with $|M \cap V_3| \ge \floor{\dd^{1/3}}$. Then we have 
\[\pr(X \textnormal{ fixates } \mid X_t=M) \ge 1 - 1/(\beta^{1/2}\dd(r-1)).\]
\end{lemma}
\begin{proof}
Recall that $\yy = \floor{(k\beta)^{1/3}}$. Let  $\xi = \floor{\dd^{1/3}}$ and 
$\TT = \floor{(r')^{(\gamma-1)/2}}$, and let $\dd_0$ be such that 
$\dd_0 \geq \max\{\beta/r,120,
{((1/\lg r')+1)}^3\}$
and, for all $\dd \ge \dd_0$,
\begin{equation}\label{eqn:b_0}
\frac{1}{{(r')}^{\xi}}
+ \frac{T}{{(r')}^{\gamma-1}}  + 
\frac{16 r^5 k^4 \beta^2 b^4 }{(r-1)(\TT-1)}  \le \frac{1}{\beta^{1/2}\dd(r-1)}.
\end{equation}
(Note that $\dd \geq \dd_0$ implicitly gives a lower bound on~$k$
since $\dd=\dd(k) \leq \sqrt{k}$.)

By Lemma~\ref{lem:gamblers-ruin}, the probability that $\ZZ^\xi$ reaches $\yy$ before zero is
\[\frac{1-(1/r')^{\xi}}{1-(1/r')^\yy} \ge 1 - \frac{1}{(r')^{ \xi}}.\]
Moreover,  
Lemma~\ref{lem:gamblers-ruin} also shows that
the probability that $\ZZ^\xi$ reaches $0$ on any given sojourn from $\yy$ is at most \[
1-\frac{1-{(\frac1{r'})}^{\gamma-1}}{1-{(\frac1{r'})}^\gamma}=
\frac{r'-1}{(r')^\yy-1} \le  \frac{1}{{(r')}^{\gamma-1}}.\]
Thus the probability that $\ZZ^\xi$ 
 never reaches zero 
 when it makes $\TT$ transitions from state $\yy$ is at least
 $$
 \left(1-\frac{1}{{(r')}^{\gamma-1}}\right)^T
 \geq 1-\frac{T}{ {(r')}^{\gamma-1}} $$
Thus the probability that $\ZZ^\xi$ reaches zero from state $\xi$ within $\TT$ transitions is at most
$$\frac{1}{{(r')}^{\xi}} + 
\frac{T}{ {(r')}^{\gamma-1}}.$$ 
If $Z^\xi$ does not reach zero within $T$ transitions
and we couple it with $X$ according to
Lemma~\ref{lem:Z-coupling},
noting that $T<I$,
then there is a $t' \geq t+T-1$
such that $X_{t'}$ is non-empty. 
Thus,
\begin{equation}\label{eqn:fixate-from-full-1}
\pr(X_{t+T-1} = \emptyset \mid X_t = M) \le \frac{1}{{(r')}^{\xi}} + 
\frac{T}{ {(r')}^{\gamma-1}}.
\end{equation}
		
Now, by Lemma~\ref{lem:absorb-time} combined with Markov's inequality, we have
\[\pr(X_{t+T-1} \notin \{\emptyset,V\} \mid X_t = M) \le \frac{r|V|^4}{(r-1)(\TT-1)} 
\leq \frac{r{(2 k r \beta^{1/2} b)}^4}{(r-1)(\TT-1)} =
\frac{16 r^5 k^4 \beta^2 b^4}{(r-1)(\TT-1)}.\] 
(Here the upper bound on $|V|$ follows from Observation~\ref{obs:incubator-nm}.) Hence by \eqref{eqn:fixate-from-full-1} and a union bound, it follows that
\begin{align*}
\pr(X_{t+T-1} \ne V \mid X_t = M) &\le \frac{1}{{(r')}^{\xi}} + 
\frac{T}{ {(r')}^{\gamma-1}}  + \frac{16 r^5 k^4 \beta^2 b^4}{(r-1)(\TT-1)}.
\end{align*}
The result therefore follows from \eqref{eqn:b_0}.
\end{proof}
	
\begin{lemma}\label{lem:fixate-from-leaf}
There exists $\dd_0$ depending only on $r$ such that the following holds whenever $\dd \ge \dd_0$. 
If $X_0 \cap V_1 \ne \emptyset$, then $(X_t)_{t \ge 0}$ fixates with probability at least $1 - 26/(\beta^{1/2}\dd(r-1))$.
\end{lemma}
\begin{proof}
By Lemma~\ref{lem:phase-one} and Lemma~\ref{lem:fixate-from-full}, when $\dd$ is sufficiently large we have
\[\pr(X\textnormal{ fixates} \mid X_0 \cap V_1 \ne \emptyset) \ge 1 - \frac{25}{\beta^{1/2}\dd(r-1)} - \frac{1}{\beta^{1/2}\dd(r-1)},\]
so the result follows.
\end{proof}

\begin{lemma}\label{lem:fixate-from-clique}
There exists $\dd_0$ depending only on $r$ such that, whenever $\dd \ge \dd_0$, for all $x_0 \in V_3$,
\[
\pr(X\textnormal{ fixates}\mid X_0 = \{x_0\}) \ge 1-2/r.
\]
\end{lemma}
\begin{proof}
Let $T = \min\{t > 0 \mid |X_t \cap V_3| =\yy\}$.
By Lemma~\ref{lem:gamblers-ruin},
the probability that $Z^1$ reaches $\yy$ before $0$ is at least $1-1/r'$. Thus by Lemma~\ref{lem:Z-coupling}, we have
\begin{equation}\label{eqn:fixate-from-clique}
\pr(T < \infty \mid X_0 = \{x_0\}) \ge 1 - 1/r'.
\end{equation}
Moreover, by Lemma~\ref{lem:fixate-from-full}, when $\dd$ is sufficiently large, for all $t > 0$ and all $M \subseteq V$ with $|M \cap V_3| \ge \yy$, we have
\[
\pr(X\textnormal{ fixates} \mid T = t, X_t = M, X_0 = \{x_0\}) \ge 1 - \frac{1}{\beta^{1/2}\dd(r-1)}.
\]
Summing over all possible values of $t$ and $M$, we obtain
\[
\pr(X\textnormal{ fixates} \mid T < \infty, X_0 = \{x_0\}) \ge 1 - \frac{1}{\beta^{1/2}\dd(r-1)}.
\]
Thus by \eqref{eqn:fixate-from-clique}, taking $\dd_0 \ge (r+1)/\sqrt{r-1}$, it follows that
\[
\pr(X\textnormal{ fixates}\mid X_0 = \{x_0\}) \ge 1 - \frac{1}{r'} - \frac{1}{\beta^{1/2}\dd(r-1)} \ge 1-\frac{2}{r+1} - \frac{1}{5r(r+1)} \ge 1 - \frac{2}{r}.
\]
\end{proof}

\subsection{Putting it all together}

We can now prove Theorem~\ref{thm:sparse-incubator},
which we restate for convenience.
\begin{thmsparseincubator}
\statethmsparseincubator{}
\end{thmsparseincubator} 
\begin{proof} 
By Lemmas~\ref{lem:fixate-from-leaf} and~\ref{lem:fixate-from-clique}, when $\dd(k)$ is sufficiently large we have
\begin{align*}
\ell_r(G) &\leq \frac{|V_3|}{n}\cdot\frac{2}{r} + \frac{|V_2|}{n} + \frac{|V_1|}{n}\left(\frac{26}{\beta^{1/2}\dd(k)(r-1)} \right) \leq \frac{3\beta k}{rn} + \frac{26}{\beta^{1/2}\dd(k)(r-1)}.
\end{align*}
Since Observation~\ref{obs:incubator-nm} implies that $n \ge kr\beta^{1/2}\dd(k)$, it follows that
\begin{equation}\label{eqn:tight-bound}
\ell_r(G) \leq \frac{3\beta^{1/2}}{r^2\dd(k)} + \frac{26}{\beta^{1/2}\dd(k)(r-1)} \le \frac{4\beta^{1/2}}{r^2\dd(k)}.
\end{equation}
Since Observation~\ref{obs:incubator-nm} implies that $m/n \le \beta^{3/2}\dd(k)/r$, it follows that
\begin{align*}
\ell_r(G) &\leq \frac{\beta^{3/2}\dd(k) n}{rm} \cdot \frac{4\beta^{1/2}}{r^2\dd(k)} = \frac{4\beta^2n}{r^3m}.
\end{align*}
Since $r^2/(r-1) > 1$, we have $\beta \le 52r^2/(r-1)$ and hence
\[
\ell_r(G) \leq \frac{2^{14}rn}{(r-1)^{2}m},
\]
as required.
\end{proof}

Finally, we prove Theorem~\ref{thm:incubator}.

\begin{thmincubator}
\statethmincubator
\end{thmincubator}

\begin{proof}
Let $\dd(k) = \floor{\sqrt{k}}$.  Consider any $r>1$ and let the constant $\dd_0$ 
(depending on~$r$) be the one from the statement of Theorem~\ref{thm:sparse-incubator}. Define
\[
\mathcal{D}_r = \{G \in \mathcal{I}_{r,\dd} \mid \mbox{The parameter, $k$, of $G$ is the square of an integer and } \dd(k) \ge \dd_0\}.
\]
Note that in the definition of $\mathcal{I}_{r,\dd}$ (Definition~\ref{def:incubator}), when $k$ is a square integer, $G[V_2,V_3]$ is a complete bipartite graph and $G[V_3]$ is a clique. Thus $\mathcal{D}_r$ is an infinite family.

Consider any $G \in \mathcal{D}_r$. Note that by Observation~\ref{obs:incubator-nm}, $n \le 2k^{3/2}r\beta^{1/2}$, and hence $k^{1/2} \ge (n/(2r\beta^{1/2}))^{1/3}$. Moreover, as in \eqref{eqn:tight-bound}, we have $\ell_r(G) \le 4\beta^{1/2}/(r^2\dd(k))$. It follows that
\[
\ell_r(G) \leq \frac{4\beta^{1/2}}{r^2}\cdot\frac{2^{1/3}r^{1/3}\beta^{1/6}}{n^{1/3}} = \frac{2^{7/3}\beta^{2/3}}{r^{5/3}n^{1/3}} \le \frac{71}{r^{1/3}(r-1)^{2/3}n^{1/3}},
\] 
and so the result follows. (The final inequality uses $\beta \le 52r^2/(r-1)$
as in the proof of Theorem~\ref{thm:sparse-incubator}.)
\end{proof}
	
\section{Lower  bounds on extinction probability}\label{sec:lowerextinct}
	
In this section, we prove Theorems~\ref{thm:digraph-lower},
\ref{thm:super-tight-upper} and~\ref{thm:sparseUB} which give lower
bounds on extinction probability.
The proofs of our theorems rely on the following quantity, which has also been studied 
in the undirected case in~\cite{MSNatural,MSStrong}.
\begin{definition}\label{defn:Q}
Given a digraph $G=(V,E)$, we define the \emph{danger}
of any vertex~$v$ as
\[
Q_v = \sum_{u\in \Nin(v)} \tfrac{1}{\dout(u)}.
\] 
\end{definition}

Note that the danger of~$v$ is essentially  
the rate at which $v$ dies when all of its in-neighbours are non-mutants. 
The following observation is immediate, since 
$Q_{u}/(r+Q_{u})$ is 
 the probability that $u$ dies before spawning a mutant when the Moran process is 
 run from state~$\{u\}$. 

\begin{observation}\label{obs:new}
Let $G = (V,E)$ be a   digraph  
with a vertex $u\in V$. 
Then
$$\ell_r(u) \geq Q_{u}/(r+Q_{u}).$$
\end{observation}

The following   lemma   gives a lower bound on $\ell_r(\{u,v\})$,
the extinction probability 
when the Moran process is run from state $\{u,v\}$. The lower bound is based on crudely
ignoring every situation except the  one in which 
the first state change is the death of the mutant at~$v$. Though this is crude, it turns out to suffice
for our purposes.
For other situations in which such arguments have been used, see Theorem~1 of~\cite{MSNatural}.

\begin{lemma}\label{lem:new}
Suppose $r \geq 1$. 
Let $G = (V,E)$ be a   digraph  
with a vertex $u\in V$ satisfying
$\ell_r(u) \leq 1/2$
and a vertex $v \in \Nout(u)$.
Then 
$$ 
\ell_r(\{u,v\}) \geq  \left(1 - \frac{3r}{2r+Q_v}\right)\ell_r(u).
$$
\end{lemma}
\begin{proof}
 Let $W= n+2(r-1)$,
 \begin{equation*}
			\overline{Q}_{u} = \sum_{w \in \Nin(u) \setminus \{v\}} \frac{1}{\dout(w)},\qquad \mbox{and}\qquad
			\overline{Q}_v = \sum_{w \in \Nin(v) \setminus \{u\}} \frac{1}{\dout(w)}.
\end{equation*}

 We consider four events which may occur when the Moran process is run, starting from state $\{u,v\}$:
 \begin{itemize}
 \item Vertex $u$ reproduces  (probability $r/W$),
 \item Vertex $v$ reproduces  (probability $r/W$),
 \item Some vertex in  $\Nin(u) \setminus \{v\}$ reproduces   onto~$u$ (probability $\overline{Q}_u/W$),
  \item Some vertex in  $\Nin(v) \setminus \{u\}$ reproduces onto~$v$ (probability $\overline{Q}_v/W$),
 \end{itemize}
 Note that any other event leaves the state unchanged. Thus,
 $\ell_r(\{u,v\}) $ is at least the probability that the  last of these happens first, before the others,
 multiplied by $\ell_r(u)$, which 
is the extinction probability 
from the resulting state (which is $\{u\}$).
Thus, 
$$\ell_r(\{u,v\})  \geq \frac{\overline{Q}_v}{2r+\overline{Q}_{u}+\overline{Q}_v}\ell_r(u).$$

Note that $\overline{Q}_{u}  \leq r$ 
(clearly $\overline{Q}_u \leq Q_u$ and Observation~\ref{obs:new}, together with $\ell_r(u) \leq 1/2$, 
implies  $Q_{u} \leq r$). Also, $\overline{Q}_v = Q_v - 1/\dout(v) \geq Q_v - r$. Hence
\[
	\ell_r(\{u,v\}) \geq \frac{\overline{Q}_v}{3r+\overline{Q}_v}\ell_r(u) = \left(1 - \frac{3r}{3r+\overline{Q}_v} \right)\ell_r(u) \geq \left(1 - \frac{3r}{2r+Q_v}\right)\ell_r(u),
\]
as required.
\end{proof}

We next use Lemma~\ref{lem:new} to
derive an upper bound on the  danger of a vertex.  
	  	 
\begin{lemma}\label{lem:fixation-props}
Suppose $r\geq 1$.
Let $G = (V,E)$ be a strongly-connected digraph with 
$|V| \geq 2$, 
and suppose that $u \in V$
satisfies $\ell_r(u) \leq 1/4$.
Then $$Q_{u} \leq \frac{4r\ell_r(u)}{\dout(u)}\sum_{v \in \Nout(u)}\frac{r}{2r+Q_v}.$$
\end{lemma}
\begin{proof}  
From the definition of the Moran process,
$$	\ell_r(u) = \frac{Q_u}{r+Q_u} + \frac{r}{r+Q_u}\cdot\frac{1}{\dout(u)}\sum_{v \in \Nout(u)} \ell_r(\{u,v\}).$$
It then follows from Lemma~\ref{lem:new} that 
\begin{align*}
\ell_r(u) &\geq \frac{Q_u}{r+Q_u} + \frac{r}{r+Q_u}\cdot\frac{1}{\dout(u)}\sum_{v \in \Nout(u)} \left(1 - \frac{3r}{2r+Q_v}\right)\ell_r(u)\\
&= \frac{Q_u}{r+Q_u} + \frac{r}{r+Q_u}\ell_r(u) - \frac{r}{r+Q_u}\cdot\frac{1}{\dout(u)}\sum_{v \in \Nout(u)} \frac{3r}{2r+Q_v}\ell_r(u).
\end{align*}
Multiplying by $r+Q_u$ and rearranging, we obtain
$$\frac{r}{\dout(u)}\sum_{v \in \Nout(u)} \frac{3r}{2r+Q_v}\ell_r(u) \geq (1-\ell_r(u))Q_u.$$
Since $\ell_r(u) \leq 1/4$, we have $3 /(1-\ell_r(u)) \leq 4$, so
$$Q_u \leq \frac{4r\ell_r(u)}{\dout(u)}\sum_{v \in \Nout(u)} \frac{r}{2r+Q_v},$$
as required.
\end{proof}
	
The following lemma gives an upper bound on the total danger of a set of vertices with low extinction probability. Throughout the rest of the section, this lemma will be  the main point of interaction 
between our arguments and
 the definition of the Moran process --- the remainder of our arguments will focus on how vertex dangers and extinction probabilities can be distributed. 	 
	 
\begin{lemma}\label{lem:danger-bound}
Let $G$ be a strongly-connected 
$n$-vertex digraph with $n \geq 2$. Consider the Moran process on $G$ with fitness $r \geq 1$. Let $S \subseteq V(G)$.
Suppose that, for some $\alpha\leq 1/4$, 
every vertex $v\in S$ has $\ell_r(v) \leq \alpha$. Then 
$\sum_{v \in S}Q_v \leq 4r^2\alpha|\Nout(S)|$
and $\sum_{v \in S}Q_v \leq 4r^2n\alpha\ell_r(G)$.
\end{lemma}
	\begin{proof}
		By applying Lemma~\ref{lem:fixation-props} to all $v \in S$,
		\begin{align*}
			\sum_{v \in S}Q_v 
			&\leq \sum_{v \in S}\frac{4r\alpha}{\dout(v)}\sum_{w \in \Nout(v)}\frac{r}{2r+Q_w} 
			= \sum_{w \in \Nout(S)}\frac{4r^2\alpha}{2r+Q_w}\sum_{v \in \Nin(w) \cap S} \frac{1}{\dout(v)}\\ 
			&\leq \sum_{w \in \Nout(S)}\frac{4r^2\alpha Q_w}{2r+Q_w} 
			\leq 4r^2\alpha\sum_{w \in \Nout(S)}\ell_r(w),
		\end{align*}
where the final inequality follows by  Observation~\ref{obs:new}. 
The first part of the result follows by bounding $\ell_r(w) \leq 1$, and the second part of the result follows since
		\[
			\sum_{w \in \Nout(S)}\ell_r(w) \leq \sum_{w \in V}\ell_r(w) = n\ell_r(G).
		\]
	\end{proof}
	
We can now prove Theorem~\ref{thm:digraph-lower}, which we restate here for convenience. Note that when $r=1$, we have $\ell_r(G) = 1 - 1/n$ for all strongly-connected $n$-vertex digraphs $G$ (see Lemma~1 of \cite{DGMRSS2014:approx}). So from now on, we will take $r > 1$.

\begin{thmdigraphlower}
\statethmdigraphlower{}
\end{thmdigraphlower} 
\begin{proof}
Let $V$ be the vertex set of~$G$. 
If $n=2$, then $\ell_r(G) = 1/(1+r) \geq 1/(5r)$ and we are done. If $n \geq 3$ and $\ell_r(G) \geq 1/8$, then likewise we are done. 
Therefore, suppose $n \geq 3$ and $\ell_r(G) < 1/8$. Note that $Q_v \geq 1/n$ for all $v \in V$. Let $A = \{v \in V \mid \ell_r(v) \leq 2\ell_r(G)\}$, and note that $\ell_r(G) > (|V \setminus A|/n)\cdot 2\ell_r(G)$ and hence
$|A| > n/2$. Applying Lemma~\ref{lem:danger-bound} to $A$ with $\alpha = 2\ell_r(G) \leq 1/4$ yields
\[\frac{1}{2} < \sum_{v \in A}Q_v \leq 8r^2n\ell_r(G)^2,\]
from which the result follows.
\end{proof}

In the proof of  Theorem~\ref{thm:digraph-lower}, we 
 used the fact that, 
in an $n$-vertex digraph, 
 every vertex~$v$ with ``low'' extinction probability has $Q_v \geq 1/n$. For undirected graphs, where we want to prove a stronger result, this bound is too loose. Instead we must 
account for the vertices  with low extinction probability  that have high danger. 
 We next show that any undirected graph with low extinction probability must contain a set of vertices with both high total degree and high minimum degree.
We will use this to prove  Theorem~\ref{thm:sparseUB} and Theorem~\ref{thm:super-tight-upper}.

\begin{lemma}\label{lem:high-deg-verts}
Let $r>1$. Consider any connected $n$-vertex graph~$G=(V,E)$ with $n\geq 2$ 
and  $\ell_r(G) \leq 1/8$. Then there exists 
a non-empty subset $B$ of $V$ such that 
$\sum_{v \in B}d(v) \geq n/(144 r^2\ell_r(G))$ and, 
for all $v \in B$, $d(v) \geq 1/(32r^2\ell_r(G)^2)$.
\end{lemma} 
		
\begin{proof} 
Let $A = \{v \in V \mid \ell_r(v) \leq 2\ell_r(G)\}$,
$A' = \{v \in A \mid Q_v <  32r^2\ell_r(G)^2\}$
and $B = N(A')$.

We first show the third claim in the statement of the lemma, that
$d(v) \geq 1/(32r^2\ell_r(G)^2)$ for all $v\in B$.
This claim follows from the fact that every $v\in B$ is adjacent to some $w\in A'$,
so $1/d(v) \leq Q_w < 32r^2\ell_r(G)^2$.

We next show that $B$ is non-empty. Since $G$ is connected and $n>1$, this follows from the fact that $A'$ is non-empty.
Instead of showing directly that $A'$ is non-empty, we will show the stronger
claim that $|A'| \geq n/4$ --- we will use this later. 
As in the proof of Theorem~\ref{thm:digraph-lower}, note that   $|A|> n/2$. 
Next, apply
Lemma~\ref{lem:danger-bound}
with $S=A$ and $\alpha = 2\ell_r(G)  \leq 1/4$.
This shows that $\sum_{v\in A} Q_v \leq 8 r^2 n\ell_r(G)^2$. Then,
from the definition of $A'$,
$$32r^2\ell_r(G)^2|A\setminus A'| \leq \sum_{v \in A\setminus A'} Q_v
\leq \sum_{v \in A} Q_v \leq 8r^2n\ell_r(G)^2,$$
so $|A\setminus A'| \leq n/4$ and
$|A'| \geq |A| - |A\setminus A'| \geq n/2 -  n/4 = n/4$. 

In the rest of the proof, the goal is to show the first
claim in the statement of the lemma, namely 
\begin{equation}
\label{eq:stateit}
\sum_{v \in B}d(v) \geq n/(144r^2\ell_r(G)).
\end{equation}
To do this, we partition $A'$.
For every positive integer~$i$, let
$Q_i = 3^{i-1}/n$ and let $A_i = \{v \in A' \mid Q_i \leq Q_v < Q_{i+1}\}$. 
Every vertex $v\in A'$ has $Q_v \geq Q_1 = 1/n$ so $A'$ is the union of the disjoint sets $A_1, A_2,\ldots$.

For every positive integer~$i$, let $B_i = N(A_i)$.
It is clear that $B= \bigcup_{i\geq 1} B_i$,  but the sets $B_i$ may not be disjoint.
 Since every vertex $v \in N(A_i)$ is joined to some vertex $w \in A_i$, we have $1/d(v) \leq Q_w < Q_{i+1}$ and hence 
\begin{equation}\label{eqn:delta-bound}
\textnormal{for all }i  \geq 1,\ \textnormal{for all }v \in B_i,\ d(v) > 1/Q_{i+1}.
\end{equation}
For $v\in B$, define $ \phi(v)$ to be the smallest $i$ such that $v \in B_i$.
Then, by~\eqref{eqn:delta-bound},
\begin{equation}\label{eq:subterms}
\sum_{v\in B} d(v) \geq \sum_{v\in B}\frac{1}{Q_{ \phi(v)+1}}. \end{equation}
For each $v\in B$ we can use the definition of $Q_i$ to obtain the following.
\begin{equation}\label{eq:nov1} \sum_{i >\phi(v) } \frac{1}{Q_{ i+1}}
= \frac{1}{Q_{\phi(v)+1}}	\sum_{j\geq  1 } \frac{1}{3^{j}}	
= 	\frac{1}{2 Q_{\phi(v)+1} }.\end{equation}
For each $v\in B$ we can  
omit indices $i$ for which $v\notin B_i$   to obtain the following.
\begin{equation}\label{eq:nov2}
 \sum_{i >\phi(v) } \frac{1}{Q_{ i+1}} \geq 
\sum_{i> \phi(v) : v \in B_i}  \frac{1}{Q_{ i+1}}	= 
\left(\sum_{i : v \in B_i } \frac{1}{Q_{ i+1}} \right) - 
 \frac{1}{Q_{\phi(v)+1} }.\end{equation}
Putting together equations~\eqref{eq:nov1} and~\eqref{eq:nov2}, we get
$$\frac{1}{Q_{\phi(v)+1}} \geq \left(\frac{2}{3}\right) \sum_{i: v \in B_i}  \frac{1}{Q_{ i+1}}.$$
Substituting this into~\eqref{eq:subterms}, we get
\begin{align}
\label{eq:sum}
\sum_{v\in B} d(v) \geq
\left(\frac{2}{3}\right) \sum_{v\in B}\sum_{i: v \in B_i} 
\frac{1}{Q_{i+1}}.\end{align}

We now apply Lemma~\ref{lem:danger-bound} 
with $S	= A_i$	and $\alpha = 2\ell_r(G) \leq 1/4$
to show that
$\sum_{v \in A_i} Q_v \leq 
4 r^2 \alpha |N(A_i)| =
8 r^2 \ell_r(G)|B_i|$.
Since, by the definition of~$A_i$, we have $Q_i \,|A_i| \leq 		\sum_{v \in A_i} Q_v$,
we conclude that 
$|B_i| \geq Q_i\,|A_i|/(8r^2\ell_r(G))$.

Equation~\eqref{eq:stateit} now follows from~\eqref{eq:sum} together with the bound
 $$\sum_{i \geq 1} \sum_{v\in B_i} \frac{1}{Q_{i+1}} \geq
   \sum_{i \geq 1}\frac{Q_i|A_i|}{8r^2Q_{i+1}\ell_r(G)} = 
   \frac{1}{24r^2\ell_r(G)}\sum_{i \geq 1}|A_i| = \frac{|A'|}
   {24r^2\ell_r(G)} \geq \frac{n}{96r^2\ell_r(G)}.
$$ 
\end{proof}

We now use
Lemma~\ref{lem:high-deg-verts}  to prove the remaining theorems.

\begin{thmsparseUB}
\statethmsparseUB{}
\end{thmsparseUB} 
\begin{proof}
Let $G=(V,E)$. Since $G$ is connected and $n \geq 2$, we have $m \geq n-1 \ge n/2$. Thus if $\ell_r(G) \geq 1/8 > n/(288 r^2m)$, then the result holds. If not, we may apply Lemma~\ref{lem:high-deg-verts} to $G$ to obtain a subset $B$ of $V$. We then have
\[
|E| \geq \frac{1}{2}\sum_{v \in B}d(v) \geq \frac{n}{288 r^2\ell_r(G)}.
\]
The result follows.\end{proof}

\begin{thmtightupper}
\statethmtightupper{} 
\end{thmtightupper}

\begin{proof}
Let $G=(V,E)$ be a connected $n$-vertex graph with $n \geq 2$. If $\ell_r(G) > 1/(15r)$ then we are done.
So suppose for the rest of the proof that $\ell_r(G) \leq 1/(15r)$.

By Lemma~\ref{lem:high-deg-verts}, factoring some of the constants to make the arithmetic easier below,
 there exists a non-empty subset $B$ of $V$ such that $\sum_{v \in B}d(v) \geq n/(2^4 3^2 r^2\ell_r(G))$ and, for all $v \in B$, $d(v) \geq 1/(2^5 r^2\ell_r(G)^2)$. For each positive integer $i$, let $d_i = 4^{i-1}/(2^5r^2\ell_r(G)^2)$ and let $B_i = \{v \in B \mid d_i \leq d(v) < d_{i+1}\}$. 
Every vertex in $B$ has degree at least $d_1$.  
Let $\mathcal{I} = \{i \geq 1 \mid |B_i| > 0\}$.
Note that the
sets in $\{B_i \mid i\in \mathcal{I}\}$ are 
disjoint and they are
a partition of~$B$.	
	
Let $D = \sum_{i \in \mathcal{I}} d_i|B_i|$. For every positive integer $i$, let $\ell_i = d_i|B_i|/(2^6 3^2r^24^{i-1}D)$.
The goal will be to show that $B_i$ sends most of its edges to vertices $v$ with $\ell_r(v)>\ell_i$, so plenty of these vertices
exist. To do this, let
$X_i = \{v \in V \mid \ell_r(v) \leq \ell_i\}$. We  start by giving an upper bound on the number of edges from $B_i$ to $X_i$. By Lemma~\ref{lem:danger-bound} applied with $S=X_i$ and $\alpha = \ell_i \leq 1/(2^6 3^2r^2) < 1/4$, for all $i \geq 1$ we have
	\begin{equation}\label{eqn:tight-upper-X-works-pre}
		\sum_{v \in X_i} Q_v \leq 4r^2n\ell_i\ell_r(G) = \frac{d_i|B_i|n\ell_r(G)}{2^4 3^2\cdot4^{i-1}D} = \frac{|B_i|n}{2^9 3^2r^2\ell_r(G)D}.
	\end{equation}
	Now, by the definitions of $d_i$, $B_i$ and $B$,
	\begin{equation}\label{eqn:tight-upper-D-bound}
		D = \sum_{i \in \mathcal{I}} d_i|B_i| = \frac{1}{4}\sum_{i \in \mathcal{I}}d_{i+1}|B_i| \ge \frac{1}{4}\sum_{v \in B}d(v)\ge \frac{n}{2^6 3^2r^2\ell_r(G)}.
	\end{equation}
	It follows from~\eqref{eqn:tight-upper-X-works-pre} and~\eqref{eqn:tight-upper-D-bound} that
	\begin{equation}\label{eqn:tight-upper-X-works}
		\sum_{v \in X_i} Q_v \le |B_i|/8.
	\end{equation}
	
	On the other hand, using the definition of $Q_v$ and then the definitions of $B_i$ and $d_i$,
	$$
	\sum_{v \in X_i}Q_v 
	\geq \sum_{v \in X_i} \sum_{w \in B_i \cap N(v)}\frac{1}{d(w)}
	= \sum_{w \in B_i}\frac{|N(w) \cap X_i|}{d(w)}
	\geq  \sum_{w \in B_i}\frac{|N(w) \cap X_i|}{d_{i+1}}
	= \frac{1}{4d_i}\sum_{w \in B_i}|N(w) \cap X_i|.
	$$
	Thus by~\eqref{eqn:tight-upper-X-works}, $\sum_{w \in B_i}|N(w) \cap X_i| \leq d_i|B_i|/2$. Now let $C_i = N(B_i) \setminus X_i$. It follows that
	\begin{equation}\label{eqn:tight-upper-avg-X}
		\sum_{w \in B_i} |N(v) \cap C_i| = \sum_{w \in B_i}d(w) - \sum_{w \in B_i}|N(w) \cap X_i| \geq d_i|B_i| - \frac{d_i|B_i|}{2} = \frac{d_i|B_i|}{2}.
	\end{equation}

Thus, we have succeeded in showing that most of the edges from~$B_i$ go to vertices 
$v\in C_i$, which have $\ell_r(v)> \ell_i$.
For the rest of the proof, we show that the vertices in the sets $C_i$
give a lower bound on $\ell_r(G)$.

	When $i \in \mathcal{I}$, $B_i \ne \emptyset$ and so there is some vertex $v \in B_i$ that sends at least as many edges into $C_i$ as the average over $B_i$. Since the degree of every vertex is an integer, it follows from~\eqref{eqn:tight-upper-avg-X} that $|C_i| \geq \ceil{d_i/2}$. Now, for all $i \in \mathcal{I}$, let $C_i'$ be an arbitrary subset of $C_i$ such that $|C_i'| = \ceil{d_i/2}$. Note that since we have assumed $\ell_r(G) \leq 1/(15r)$, we have $d_1 = 1/(2^5 r^2\ell_r(G)^2) \geq 6$. Thus for all $i \in \mathcal{I}$, $d_i \geq d_1 \geq 6$ and so 
	\begin{equation}\label{eqn:tight-upper-Ci}
		d_i/2 \leq |C_i'| \leq 2d_i/3.
	\end{equation}
	Let $C_i'' = C_i' \setminus \bigcup_{j \in [i-1] \cap \mathcal{I}} C_j'$. Then all sets $C_i''$ are disjoint by construction, and by~\eqref{eqn:tight-upper-Ci} we have
	\[
		|C_i''| \geq \frac{d_i}{2} - \sum_{j \in [i-1] \cap \mathcal{I}} |C_j'| 
		\geq \frac{d_i}{2} - \sum_{j \in [i-1]}\frac{2d_j}{3} 
		= \frac{d_i}{2} - \frac{2}{3}\sum_{j \in [i-1]} \frac{d_i}{4^j}
		\geq d_i\left(\frac{1}{2} - \frac{2}{3}\sum_{j=1}^\infty 4^{-j} \right)
		> \frac{d_i}{4}.
	\]
	Since $C_i'' \subseteq C_i' \subseteq C_i = N(B_i) \setminus X_i$, we have $C_i'' \cap X_i = \emptyset$. So by the definition of $X_i$, for all $v \in C_i''$, we have $\ell_r(v) > \ell_i$. It follows that
	\begin{align*}
		\ell_r(G) &= \frac{1}{n}\sum_{v \in V}\ell_r(v) 
		\geq \frac{1}{n}\sum_{i \in \mathcal{I}}|C_i''|\ell_i
		\geq \frac{1}{n}\sum_{i \in \mathcal{I}}\frac{d_i}{4} \cdot \frac{d_i|B_i|}{2^6 3^2r^24^{i-1}D}\\
		&= \frac{1}{n}\sum_{i \in \mathcal{I}}\frac{4^{i-1}}{2^7 r^2\ell_r(G)^2} \cdot \frac{d_i|B_i|}{2^6 3^2 r^24^{i-1}D}
		= \frac{\sum_{i \in \mathcal{I}}d_i|B_i|}{2^{13}3^2r^4Dn\ell_r(G)^2}
		= \frac{1}{2^{13}3^2r^4n\ell_r(G)^2}.
	\end{align*}
	Rearranging yields $\ell_r(G) \geq 1/(42r^{4/3}n^{1/3})$, as required.
\end{proof}

\section*{Acknowledgements}We thank an anonymous referee of an earlier version for suggesting a 
simplification that    shortened the proof of Lemma~\ref{lem:high-deg-verts}.

\bibliographystyle{plain}
\bibliography{\jobname}

\end{document}